\documentclass[final]{siamart1116}
\numberwithin{theorem}{section}
\numberwithin{equation}{section}
\usepackage[all]{xy}
\usepackage{amssymb}
\usepackage{braket,amsfonts,amsopn}
\usepackage{bm}
\usepackage{makecell,multirow,diagbox}
\usepackage{mathrsfs}
\usepackage{graphicx,epstopdf}
\usepackage{subfigure}
\newsiamremark{hypothesis}{Hypothesis}
\newsiamremark{remark}{Remark}
\DeclareMathOperator{\divg}{div}
\DeclareMathOperator{\curlvect}{\overrightarrow{\nabla}\times}
\DeclareMathOperator{\I}{I}
\DeclareMathOperator{\rot}{\mathbf{rot}}
\DeclareMathOperator{\grad}{grad}
\DeclareMathOperator{\dist}{dist}
\newcommand{\curl}[1]{\overrightarrow{\nabla}\times#1}
\newcommand{\df}[1]{\partial#1}
\newcommand{\dftk}[1]{\frac{\partial#1}{\partial\bm{t}_{k}}}
\newcommand{\dfnk}[1]{\frac{\partial#1}{\partial\bm{n}_{k}}}
\newcommand{\dftkp}[1]{\frac{\partial#1}{\partial\bm{t}_{k+1}}}
\newcommand{\dftkm}[1]{\frac{\partial#1}{\partial\bm{t}_{k-1}}}
\newcommand{\dfte}[1]{\frac{\partial#1}{\partial\bm{t}_{e}}}

\newcommand{\Rom}[1]{\uppercase\expandafter{\romannumeral#1}}

\usepackage{lipsum}
\usepackage{amsfonts}
\usepackage{graphicx}
\usepackage{epstopdf}
\usepackage{algorithmic}
\ifpdf
  \DeclareGraphicsExtensions{.eps,.pdf,.png,.jpg}
\else
  \DeclareGraphicsExtensions{.eps}
\fi

\newcommand{\TheTitle}{Global superconvergence of the lowest order mixed finite element on mildly structured meshes} 
\newcommand{\TheTitleShort}{Superconvergence of the lowest order mixed element} 
\newcommand{\TheAuthors}{Yu-Wen Li}

\headers{\TheTitleShort}{\TheAuthors}

\title{{\TheTitle}}

\author{
  Yu-Wen Li\thanks{Department of Mathematics, University of California, San Diego, La Jolla, 
CA 92093. \email{yul739@ucsd.edu}.}
}

\usepackage{amsopn}

\ifpdf
\hypersetup{
  pdftitle={\TheTitle},
  pdfauthor={\TheAuthors}
}
\fi

\begin{document}

\maketitle

\begin{abstract}
In this paper, we develop global superconvergence estimates for the lowest order Raviart--Thomas mixed finite element method for second
order elliptic equations with general boundary conditions on triangular meshes, where most pairs of adjacent triangles form approximate parallelograms. In particular, we prove the $L^{2}$-distance between the numerical solution and canonical interpolant for the vector variable
is of order $1+\rho$, where $\rho\in(0,1]$ is dependent on the mesh structure. By a cheap local postprocessing operator $G_{h}$, we prove the $L^{2}$-distance between the exact solution and the postprocessed numerical solution for the vector variable is of order $1+\rho$. As a byproduct, we also obtain the superconvergence estimate for Crouzeix--Raviart nonconforming finite elements on triangular meshes of the same type. 
\end{abstract}



\section{Introduction}\label{intro}
Let $\Omega\subset\mathbb{R}^{2}$ be a bounded domain with Lipschitz boundary $\partial\Omega$. For simplicity of presentation, we
assume $\Omega$ is a polygon. The Sobolev seminorms and norms are defined by
\begin{equation*}
\begin{aligned}
&|u|_{k,p,\Omega}=\left(\sum_{|\alpha|=k}\int_{\Omega}|\partial^{\alpha}u|^{p}\right)^{\frac{1}{p}},\quad
||u||_{k,p,\Omega}=\left(\sum_{l=0}^{k}|u|_{l,p,\Omega}^{p}\right)^{\frac{1}{p}},\\
&|u|_{k,\Omega}=|u|_{k,2,\Omega},\quad||u||_{k,\Omega}=||u||_{k,2,\Omega}.
\end{aligned}
\end{equation*}
Sobolev norms with $\infty$-index, norms of vector/matrix-valued functions, and
fractional order norms are generalized in usual ways.

We consider the following second order elliptic equation:
\begin{subequations}\label{mix:c1}
\begin{align}
&-\divg(\bm{A}(\bm{x})\nabla u+\bm{b}(\bm{x})u)+c(\bm{x})u=f(\bm{x}),\quad\bm{x}\in\Omega,\\
&u=g(\bm{x}),\quad\bm{x}\in\partial\Omega,
\end{align}
\end{subequations}
where $\bm{A}$ is symmetric and uniformly elliptic, $\bm{A}, \bm{b}, c$ are sufficiently smooth on $\overline{\Omega}$.
Let
\begin{equation*}
\bm{p}=\bm{A}(\bm{x})\nabla u+\bm{b}(\bm{x})u
\end{equation*}
and set
\begin{equation*}
\bm{\alpha}=\bm{A}(\bm{x})^{-1},\quad\bm{\beta}=\bm{\alpha}(\bm{x})\bm{b}(\bm{x}).
\end{equation*} 
\cref{mix:c1} can be written in the form of the first order system:
\begin{subequations}\label{mix:c}
\begin{align}
&\bm{\alpha}\bm{p}-\bm{\beta}u-\nabla u=0,\quad\bm{x}\in\Omega,\\
&-\divg\bm{p}+cu=f,\quad\bm{x}\in\Omega,\\
&u=g,\quad\bm{x}\in\partial\Omega.
\end{align}
\end{subequations}

Denote
\begin{equation*}
\mathcal{Q}=\{\bm{q}\in L^{2}(\Omega)^{2}:
\divg\bm{q}\in L^{2}(\Omega)\},\quad\mathcal{V}=L^{2}(\Omega).
\end{equation*}
Let $(\cdot,\cdot)$ and $\langle\cdot,\cdot\rangle$ denote the $L^{2}(\Omega)$ and $L^{2}(\partial\Omega)$ inner product, respectively.
Let $\bm{n}$ denote the outward unit normal to $\partial\Omega$. The mixed formulation for \cref{mix:c} is to find $\{\bm{p},u\}\in\mathcal{Q}\times\mathcal{V}$, such that
\begin{subequations}\label{mix:v}
\begin{align}
&(\bm{\alpha}\bm{p},\bm{q})-(\bm{q},\bm{\beta}u)+(\divg\bm{q},u)=\langle\bm{q}\cdot\bm{n},g\rangle,\label{mix:va}\\
&-(\divg\bm{p},v)+(cu,v)=(f,v),
\end{align}
\end{subequations}
for each pair $\{\bm{q},v\}\in\mathcal{Q}\times\mathcal{V}$.
Let $\{\mathcal{T}_{h}\}$ be a family of triangulations of $\Omega$, where $0<h<1$ is the mesh size. Let $\mathcal{P}_{p}(\tau)$
denote the set of polynomials of degree $\leq p$
on $\tau$. Denote
\begin{equation}\label{formRT}
\mathcal{RT}_{0}(\tau):=\{\bm{a}+a\bm{x}: \bm{a}\in\mathbb{R}^{2},\ a\in\mathbb{R}\}.
\end{equation}
The lowest order Raviart--Thomas (RT) finite element spaces are defined by
\begin{displaymath}
\begin{aligned}
&\mathcal{Q}_{h}:=\left\{\bm{q}_{h}\in\mathcal{Q}: \bm{q}_{h}|_{\tau}\in\mathcal{RT}_{0}(\tau),\ \forall\tau\in\mathcal{T}_{h}\right\},\\
&\mathcal{V}_{h}:=\{v_{h}\in\mathcal{V}: v_{h}|_{\tau}\in\mathcal{P}_{0}(\tau),\ \forall\tau\in\mathcal{T}_{h}\},
\end{aligned}
\end{displaymath}
The mixed finite element approximation to the problem \cref{mix:v} is to find $\{\bm{p}_{h},u_{h}\}\in\mathcal{Q}_{h}
\times\mathcal{V}_{h}$, such that
\begin{subequations}\label{mix:dv}
\begin{align}
&(\bm{\alpha}\bm{p}_{h},\bm{q}_{h})-(\bm{q}_{h},\bm{\beta}u_{h})
+(\divg\bm{q}_{h},u_{h})=\langle\bm{q}_{h}\cdot\bm{n},g\rangle,\quad\bm{q}_{h}\in\mathcal{Q}_{h},\label{mix:dva}\\
&-(\divg\bm{p}_{h},v_{h})+(cu_{h},v_{h})=(f,v_{h}),\quad v_{h}\in\mathcal{V}_{h}.
\end{align}
\end{subequations}
Under the assumption that \cref{mix:c} is solvable for $\{f, g\}\in L^{2}(\Omega)\times H^{\frac{3}{2}}(\Omega)$ and that
\begin{equation}\label{fullregularity}
||u||_{2,\Omega}\lesssim||f||_{0,\Omega}+||g||_{\frac{3}{2},\Omega},
\end{equation}
Douglas and Roberts \cite{Douglas1985} proved the well-posedness and a priori error estimates for the method \cref{mix:dv}.

In this paper, we shall prove supercloseness/superconvergence  results for $||\Pi_{h}\bm{p}-\bm{p}_{h}||_{0,\Omega}$
and $||\divg(\Pi_{h}\bm{p}-\bm{p}_{h})||_{0,\Omega}$, where $\Pi_{h}$
and $P_{h}$ are the interpolation operators for the lowest order RT element. In particular, we shall prove that
\begin{subequations}\label{superup}
\begin{align}
&||\Pi_{h}\bm{p}-\bm{p}_{h}||_{0,\Omega}\lesssim
h^{1+\rho}||u||_{4+\varepsilon,\Omega},\quad\varepsilon>0,\label{resultp}\\
&||\divg(\Pi_{h}\bm{p}-\bm{p}_{h})||_{0,\Omega}\lesssim h^{2}||u||_{3,\Omega},\label{superdivintro}
\end{align}
\end{subequations}
where $\rho=\min(1,\alpha,\sigma/2)$.
\cref{superdivintro} holds on general shape regular meshes while \cref{resultp} holds on quasi-uniform $\{\mathcal{T}_{h}\}$ satisfying the
piecewise $(\alpha,\sigma)$-condition. The $(\alpha,\sigma)$-grid or its simplified versions have been considered by many authors 
(cf. \cite{Lin1985,LMW2000,BaXu2003,HuangXu2008,Xu2003} and references therein). Roughly speaking, $\mathcal{T}_{h}$ is said to be 
an $(\alpha,\sigma)$-grid if most pairs of adjacent triangles in $\mathcal{T}_{h}$ form
$\mathcal{O}(h^{1+\alpha})$ approximate parallelograms except for a region of measure $\mathcal{O}(h^{\sigma})$
(cf. \cref{alphasigma}). \cref{resultp} has several generalizations. For example, the quasi-uniformity assumption can be removed under
the pure Neumann boundary condition or at the expense of a slower superconvergence rate $\rho$, see \cref{sec:varerr,superclosepNeumann} for details.

\cref{superup} is closely related to the superconvergence of the finite element solution to the exact solution. For example, by postprocessing $\bm{p}_{h}$ by a simple local averaging operator $G_{h}$ proposed in \cite{Brandts1994}, we achieve the following superconvergence estimate:
\begin{equation}\label{superpost}
||\bm{p}-G_{h}\bm{p}_{h}||_{0,\Omega}\lesssim
h^{1+\rho}||u||_{4+\varepsilon,\Omega}.
\end{equation}
The recovered flux $G_{h}\bm{p}_{h}$ can be used to develop a posteriori error estimates.
Due to the superconvergence \cref{superpost}, $||G_{h}\bm{p}_{h}-\bm{p}_{h}||_{0,\Omega}$ is known to be
an asymptotically exact a posteriori estimator for
$||\bm{p}-\bm{p}_{h}||_{0,\Omega}$ (cf. \cite{Brandts1994,BaXu2003II,Ainsworth2000}), that is,
\begin{equation*}
\lim_{h\to0}\frac{||G_{h}\bm{p}_{h}-\bm{p}_{h}||_{0,\Omega}}{||\bm{p}-\bm{p}_{h}||_{0,\Omega}}=1.
\end{equation*}
As a byproduct, \cref{superpost} also gives the following superconvergence estimate for Crouzeix--Raviart (CR) nonconforming
finite elements (cf. \cref{superCR}):
\begin{equation}\label{superpostCR}
||\nabla u-G_{h}\nabla_{h}u_{h}^{CR}||_{0,\Omega}\lesssim h^{1+\rho}||u||_{4+\varepsilon,\Omega},
\end{equation}
where $u_{h}^{CR}$ is the CR finite element solution of Poisson's equation.

The study of supercloseness between the finite element interpolant and finite element solution has a long history. For the analogue of \cref{resultp} for standard Lagrange elements on mildly structured grids, see \cite{BaXu2003,Xu2003,HuangXu2008} and references therein. For superconvergence of the scalar variable $u$ in mixed methods, see \cite{Arnold1985,Brezzi1985,Wang1989} and references therein. In practice, it is frequently the case that the vector variable
$\bm{p}$ is more important than the scalar $u$. Superconvergence results of rectangular/quadrilateral mixed finite elements for the vector variable $\bm{p}$ are well established (cf. \cite{Duran1990,Ewing1991,Ewing1999}). However, corresponding superconvergence theory of triangular mixed finite elements are much less sophisticated. To our best knowledge, the only proven superconvergence estimate of triangular mixed elements for the vector variable 
are in \cite{Douglas1983,Brandts1994,Brandts2000}. In \cite{Douglas1983}, the authors  postprocessed $\bm{p}_{h}$
and achieved interior superconvergence by convolution with a Bramble--Schatz kernel \cite{Bramble1977} which is constructed on
uniform grids, i.e. in the case of $\alpha=\sigma=\infty$. 
For the lowest order RT element on uniform grids in the case that $\bm{b}=\bm{0}, c=0$ in \cref{mix:c1}, Brandts \cite{Brandts1994} proved
\begin{equation}\label{resultB}
||\Pi_{h}\bm{p}-\bm{p}_{h}||_{0,\Omega}\lesssim h^{\frac{3}{2}}(||\bm{p}||_{\frac{3}{2},\Omega}
+h^{\frac{1}{2}}|\bm{p}|_{1,\Omega}+h^{\frac{1}{2}}|\bm{p}|_{2,\Omega}),
\end{equation}
In \cite{Brandts2000}, he also proved an analogue of \cref{resultB} for second order RT elements on uniform grids in the case that $\bm{A}=I_{2\times2}, \bm{b}=\mathbf{0}, c=0$.

Our result \cref{superup} improves existing results significantly in several ways.  First, our estimate holds on general mildly structured grids instead of uniform grids. As pointed out in \cite{BaXu2003,Xu2003}, the $(\alpha,\sigma)$-condition is very flexible and satisfied by many mature finite element codes. Second, in the best case that $\rho=1$, \cref{resultp} becomes
\begin{equation*}
||\Pi_{h}\bm{p}-\bm{p}_{h}||_{0,\Omega}\lesssim h^{2}||u||_{4+\varepsilon,\Omega},
\end{equation*}
which shows that the estimate \cref{resultB} is suboptimal. This improvement results from carefully handling the boundary error, which is usually the trickiest part in global superconvergence estimates if test functions have nonzero trace. In addition, due to the cancellation of errors on boundary elements, \cref{resultp} holds on not only $(\alpha,\sigma)$-grids but also piecewise  $(\alpha,\sigma)$-grids 
(cf. \cref{defpw,piecewise}). Third, our superconvergence results allow the convection term $\bm{b}(x)\cdot\nabla u$ and reaction term $c(x)u$. Unlike the case of the standard variational formulation for elliptic equations, the error analysis of mixed methods with nonvanishing $\bm{b}, c$ is much more involved than the case $\bm{b}=\bm{0}, c=0$ (cf. \cite{Douglas1985} and \cref{bc=0}). Last, the superconvergence estimate \cref{superpostCR} for CR nonconforming elements is obtained. Since $u_{h}^{CR}$ has jump on each interior edge, it is very difficult to prove superconvergence of nonconforming methods on triangular grids directly (cf. \cite{HuMa2016} and references therein).

The key ingredient of the proof of \cref{resultp} is two fold. First, we develop the variational error expansion for RT elements on a local triangle in terms of $\bm{q}_{h}\cdot\bm{n}_{k}$, the normal trace of $\bm{q}_{h}\in\mathcal{Q}_{h}$ on $e_{k}$, where 
$\{e_{k}\}_{k=1}^{3}$ are three edges of the triangle and $\bm{n}_{k}$ is the outward unit normal to $e_{k}$.
Due to the continuity of $\bm{q}_{h}\cdot\bm{n}_{k}$ on $e_{k}$ and the $(\alpha,\sigma)$-condition, the lower order global variational error associated with interior edges is canceled in a very delicate and transparent way instead of using soft analysis tools (the Bramble--Hilbert lemma etc., cf. \cite{Brandts1994}). The aforementioned basic idea is motivated by Bank and Xu \cite{BaXu2003}. But the technicality here is quite different because of the apparent difference between Lagrange elements and RT elements. Second, we split $\Pi_{h}\bm{p}-\bm{p}_{h}$ into two parts by the discrete Helmholtz decomposition \cref{disHelmholtz}. The norm of one part can be estimated by \cref{superdivintro,supercloseu}. To obtain optimal order global superconvergence, the error associated with another part occurring on triangles near the boundary is treated carefully by the Sobolev and discrete Sobolev inequalities, see \cref{sec:varerr} for details.

The rest of this paper is organized as follows: \cref{sec:preli} contains technical geometric identities and local error expansions. In \cref{sec:varerr}, we estimate the global variational error that forms a basis for the estimate \cref{superup}. The superconvergence result \cref{superup} and related results are presented in \cref{sec:superp}. In \cref{sec:post}, we  develop the superconvergence estimate \cref{superpost} and the related estimate \cref{superpostCR} for CR nonconforming elements. In \cref{sec:numerexp} we present a few numerical examples illustrating the optimality and flexibility of our estimates. 

\section{Preliminaries}\label{sec:preli}
\begin{figure}[tbhp]
\centering
\includegraphics[width=13.0cm,height=5.0cm]{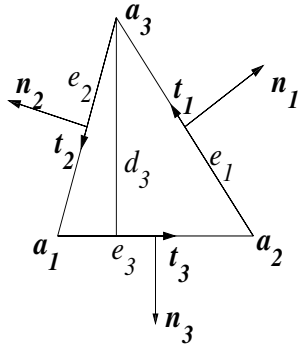}
\caption{a local triangle and associated quantities}
\label{triangle}
\end{figure}
We begin with geometric identities on a local element $\tau$. It has three vertices $\{\bm{a}_{k}\}_{k=1}^{3}$, oriented counterclockwise,
and corresponding barycentric coordinates $\{\lambda_{k}\}_{k=1}^{3}$.
Let $e_{k}$ denote the edge opposite to $\bm{a}_{k}$, $\theta_{k}$ the angle opposite to $e_{k}$, $\ell_{k}$ the length of $e_{k}$, $d_{k}$ the distance from $\bm{a}_{k}$ to $e_{k}$,
$\bm{t}_{k}$ the unit tangent to $e_{k}$, oriented counterclockwise, $\bm{n}_{k}$ the unit outward normal
to $e_{k}$, see \cref{triangle}. Corresponding quantities on $\tau^{\prime}$ and $\tau^{\prime\prime}$ have superscripts $\prime$ and $\prime\prime$ respectively. The subscripts are equivalent mod 3.
From Bank and Xu \cite{BaXu2003}, we have the following identities:
\begin{subequations}
\begin{align}
&\bm{t}_{k}=\frac{\cos\theta_{k+1}}{\sin\theta_{k}}\bm{n}_{k+1}-
\frac{\cos\theta_{k-1}}{\sin\theta_{k}}\bm{n}_{k-1},\label{decomp1}\\
&\bm{n}_{k-1}=-\sin\theta_{k+1}\bm{t}_{k}-\cos\theta_{k+1}\bm{n}_{k},\quad
\bm{t}_{k-1}=-\cos\theta_{k+1}\bm{t}_{k}+\sin\theta_{k+1}\bm{n}_{k},\label{decomp2}\\
&\bm{n}_{k+1}=\sin\theta_{k-1}\bm{t}_{k}-\cos\theta_{k-1}\bm{n}_{k},\quad
\bm{t}_{k+1}=-\cos\theta_{k-1}\bm{t}_{k}-\sin\theta_{k-1}\bm{n}_{k},\label{decomp3}\\
&\sin\theta_{k}\int_{e_{k+1}}v\lambda_{k}\lambda_{k-1}=\sin\theta_{k+1}\int_{e_{k}}v\lambda_{k+1}\lambda_{k-1}
-\int_{\tau}\dftkm{v}(1-\lambda_{k-1})\lambda_{k-1},\label{trans1}\\
&\sin\theta_{k}\int_{e_{k-1}}v\lambda_{k}\lambda_{k+1}=\sin\theta_{k-1}\int_{e_{k}}v\lambda_{k+1}\lambda_{k-1}
+\int_{\tau}\dftkp{v}(1-\lambda_{k+1})\lambda_{k+1},\label{trans2}\\
&\nabla\lambda_{k}=-\bm{n}_{k}/d_{k}.
\end{align}
\end{subequations}

In addition, we have two planar curl operators
\begin{displaymath}
\overrightarrow{\nabla}\times v=\left(\frac{\partial v}{\partial x_{2}},-\frac{\partial v}{\partial x_{1}}\right)^{t},
\quad\nabla\times\bm{q}=\frac{\partial q_{2}}{\partial x_{1}}-\frac{\partial q_{1}}{\partial x_{2}}.
\end{displaymath}
For convenience, we define the matrix
\begin{equation}\label{eq:matrices}
\rot=\begin{bmatrix}0&-1\\1&0\end{bmatrix}.
\end{equation}
It's clear that $\rot$ rotates a vector by degree $\pi/2$ counterclockwise. By direct calculation, we have the following identities:
\begin{subequations}
\begin{align}
&\rot\bm{n}_{k}=\bm{t}_{k},\quad\rot\bm{t}_{k}=-\bm{n}_{k},\label{rottk}\\
&\nabla=\rot\curl,\quad\nabla\times=\divg\rot^{-1},\label{curl:grad}\\
&\curl(vw)=v\overrightarrow{\nabla}\times w+w\overrightarrow{\nabla}\times v,\label{lib}\\
&\nabla\times(v\bm{q})=-(\curl v)\cdot\bm{q}+v\nabla\times\bm{q},\\
&\int_{\tau}v\nabla\times\bm{q}=\sum_{k=1}^{3}\int_{e_{k}}v\bm{q}\cdot\bm{t}_{k}+\int_{\tau}\curl{v}\cdot\bm{q},\label{ipt}\\
&\curl\lambda_{i}=\bm{t}_{i}/d_{i}.
\end{align}
\end{subequations}

Now we introduce basic definitions for RT elements. On the element $\tau$, the degrees of freedom of the lowest order RT elements are defined by
\begin{equation*}
\mathcal{N}_{k}(\bm{q})=\int_{e_{k}}\bm{q}\cdot\bm{n}_{k},\quad1\leq k\leq3.
\end{equation*}
For $\bm{q}\in\mathcal{Q}$, the interpolant $\Pi_{h}\bm{q}$ is the element in $\mathcal{Q}_{h}$ whose restriction to $\tau$ is the unique element in $\mathcal{RT}_{0}(\tau)$ such that
\begin{equation}\label{def:pi}
\mathcal{N}_{k}(\Pi_{h}\bm{q})=\mathcal{N}_{k}(\bm{q}),\quad1\leq k\leq3.
\end{equation}
For $v\in\mathcal{V}$, the interpolant $P_{h}v$ is the $L^{2}(\Omega)$-projection of $v$ onto $\mathcal{V}_{h}$.
$P_{h}$ and $\Pi_{h}$ are connected by the following commuting diagram, which is crucial to the stability and error analysis of mixed methods (cf. \cite{RT1977}).
\begin{equation}\label{com:diag}
\xymatrix{
    \mathcal{Q} \ar[rr]^{\divg}\ar[d]_{\Pi_{h}} & & \mathcal{V} \ar@{->}[d]^{P_{h}} \\
    \mathcal{Q}_{h} \ar[rr]^{\divg}& & \mathcal{V}_{h} \ar@{->}[rr]&&0
    }
\end{equation}
In addition, the following approximation properties hold:
\begin{subequations}\label{approxRT}
\begin{align}
&||\bm{q}-\Pi_{h}\bm{q}||_{0,\Omega}\lesssim h||\nabla_{h}\bm{q}||_{0,\Omega},\label{approxRTa}\\
&||\divg(\bm{q}-\Pi_{h}\bm{q})||_{0,\Omega}\lesssim h||\nabla_{h}\divg\bm{q}||_{0,\Omega},\\
&||v-P_{h}v||_{0,\Omega}\lesssim h||\nabla_{h}v||_{0,\Omega}.
\end{align}
\end{subequations}
where $\nabla_{h}$ is the piecewise gradient.

The following facts can be checked in a straightforward way:
\begin{equation*}
\phi_{k}=\lambda_{k+1}\overrightarrow{\nabla}\times\lambda_{k-1}
-\lambda_{k-1}\overrightarrow{\nabla}\times\lambda_{k+1},\quad1\leq k\leq3,
\end{equation*}
is the dual basis of $\{\mathcal{N}_{k}\}_{k=1}^{3}$. $\{\phi_{k}\}_{k=1}^{3}$ in Cartesian coordinate are
\begin{equation}\label{Cartesian}
\phi_{k}=\frac{\bm{x}-\bm{a}_{k}}{2|\tau|},\quad1\leq k\leq3.
\end{equation}
$\{\phi_{k}\}_{k=1}^{3}$ together with
\begin{equation}\label{psi:j}
\psi_{k}=\lambda_{k+1}\overrightarrow{\nabla}\times\lambda_{k-1}+\lambda_{k-1}\overrightarrow{\nabla}\times\lambda_{k+1},
\quad 1\leq k\leq3,
\end{equation}
form a basis of $\mathcal{P}_{1}(\tau)^{2}$. $\mathcal{N}_{i}$ vanishes at $\psi_{j}$ for $1\leq i,j\leq3$.

It turns out that  the CR interpolation is very useful in the analysis of RT elements. For $\bm{q}\in H^{1}(\tau)^{2}$,
the local CR interpolant $\I_{h}^{CR}\bm{q}$ on $\tau$ is the unique element in $\mathcal{P}_{1}(\tau)^{2}$ such that
\begin{equation}\label{CRinterp}
\int_{e_{k}}\I_{h}^{CR}\bm{q}=\int_{e_{k}}\bm{q},\quad 1\leq k\leq3.
\end{equation}
In addition, $\I_{h}^{CR}$ and $\Pi_{h}$ are connected by the following lemma.
\begin{lemma}\label{interp:cr:rt}
\begin{equation*}
\Pi_{h}\I_{h}^{CR}\bm{q}=\Pi_{h}\bm{q}.
\end{equation*}
\end{lemma}
\begin{proof}
It follows from \cref{def:pi,CRinterp} that
\begin{equation*}
\mathcal{N}_{k}(\Pi_{h}\I_{h}^{CR}\bm{q}-\Pi_{h}\bm{q})
=\int_{e_{k}}(\I_{h}^{CR}\bm{q}-\bm{q})\cdot\bm{n}_{k}=0.
\end{equation*}
\cref{interp:cr:rt} is then from the unisolvence of RT elements.
\end{proof}

Now, we expand the interpolation error for linear functions.
\begin{lemma}\label{er:exp}
For $\bm{p}_{L}\in\mathcal{P}_{1}(\tau)^{2}$,
\begin{equation*}
\bm{p}_{L}-\Pi_{h}\bm{p}_{L}=\overrightarrow{\nabla}\times r,
\end{equation*}
where
\begin{equation*}
r=-\sum_{k=1}^{3}\frac{\ell_{k}^{2}}{2}\bm{n}_{k}\cdot\dftk{\bm{p}_{L}}\lambda_{k-1}\lambda_{k+1}.
\end{equation*}
\end{lemma}
\begin{proof}
First, it follows from $\mathcal{N}_{i}(\psi_{k})=0$ and $\mathcal{N}_{i}(\bm{p}_{L}-\Pi_{h}\bm{p}_{L})=0$ that
\begin{equation*}
\bm{p}_{L}-\Pi_{h}\bm{p}_{L}=\sum_{k=1}^{3}\alpha_{k}\psi_{k}.
\end{equation*}
Then by \cref{psi:j,curl:grad,lib}, we arrive at
\begin{equation}\label{pl}
\rot(\bm{p}_{L}-\Pi_{h}\bm{p}_{L})=
\nabla\left(\sum_{k=1}^{3}\alpha_{k}\lambda_{k-1}\lambda_{k+1}\right)=\nabla r.
\end{equation}
It remains to verify that
\begin{equation}\label{goal}
\alpha_{k}=-\frac{\ell_{k}^{2}}{2}\bm{n}_{k}\cdot\dftk{\bm{p}_{L}}.
\end{equation}
Taking inner products with $\bm{t}_{k}$  and then taking the directional derivative along $\bm{t}_{k}$ on both sides of  \cref{pl} leads to
\begin{equation}\label{cond:0}
\bm{n}_{k}\cdot\left(\dftk{\bm{p}_{L}}-\dftk{\Pi_{h}\bm{p}_{L}}\right)=\frac{\partial^{2}r}{\partial\bm{t}_{k}^{2}}.
\end{equation}
The definition of $\mathcal{RT}_{0}(\tau)$ \cref{formRT} implies that $\partial\Pi_{h}\bm{p}_{L}/\partial\bm{t}_{k}$
is parallel to $\bm{t}_{k}$ and therefore
\begin{equation}\label{cond:1}
\bm{n}_{k}\cdot\dftk{\Pi_{h}\bm{p}_{L}}=0.
\end{equation}
For the right hand side,
\begin{equation}\label{cond:2}
\frac{\df^{2}r}{\df\bm{t}_{k}^{2}}=2\alpha_{k}\dftk{\lambda_{k-1}}\dftk{\lambda_{k+1}}=-\frac{2\alpha_{k}}{\ell_{k}^{2}}.
\end{equation}
Combing \cref{cond:0}, \cref{cond:1} and \cref{cond:2}, we obtain \cref{goal}.
\end{proof}

\section{Variational error expansions}\label{sec:varerr}
The following is our main technical lemma for estimating the global variational error of mixed methods.
\begin{lemma}\label{err1}
For $\bm{q}_{h}\in\mathcal{P}_{0}(\tau)$,
\begin{equation}\label{locallemmaeq}
\int_{\tau}(\bm{p}_{L}-\Pi_{h}\bm{p}_{L})\cdot\bm{q}_{h}=\sum_{k=1}^{3}
\cot\theta_{k}\int_{e_{k}}\lambda_{k-1}\lambda_{k+1}\left(\sum_{j=1}^{3}
\alpha_{k}^{(j)}\mathcal{A}_{k}^{(j)}\bm{p}_{L}\right)\bm{q}_{h}\cdot\bm{n}_{k},
\end{equation}
where
\begin{equation}\label{alpha}
\alpha_{k}^{(1)}=|\tau|,\quad\alpha_{k}^{(2)}=-|\tau|,\quad\alpha_{k}^{(3)}
=\frac{1}{2}(\ell_{k-1}^{2}-\ell_{k+1}^{2}),
\end{equation}
and $\mathcal{A}_{k}^{(j)}$ are operators defined by
\begin{equation}\label{operatorA}
\mathcal{A}_{k}^{(1)}=\bm{t}_{k}\cdot\dftk{},\quad\mathcal{A}_{k}^{(2)}=\bm{n}_{k}\cdot\dfnk{},
\quad\mathcal{A}_{k}^{(3)}=\bm{n}_{k}\cdot\dftk{}.
\end{equation}
\end{lemma}
\begin{proof}
Using \cref{ipt,er:exp}, we have
\begin{equation*}
\int_{\tau}(\bm{p}_{L}-\Pi_{h}\bm{p}_{L})\cdot\bm{q}_{h}
=-\sum_{k=1}^{3}\int_{e_{k}}r\bm{q}_{h}\cdot\bm{t}_{k}.
\end{equation*}
Therefore, it follows from \cref{decomp1,trans1,trans2} that
\begin{equation*}
\begin{aligned}
&\quad\int_{\tau}(\bm{p}_{L}-\Pi_{h}\bm{p}_{L})\cdot\bm{q}_{h}\\
&=\sum_{k=1}^{3}\int_{e_{k}}\frac{\ell_{k}^{2}}{2}\bm{n}_{k}\cdot\dftk{\bm{p}_{L}}
\lambda_{k-1}\lambda_{k+1}\bm{q}_{h}\cdot\bm{t}_{k}\\
&=\sum_{k=1}^{3}\int_{e_{k}}\frac{\ell_{k}^{2}}{2}\bm{n}_{k}\cdot\dftk{\bm{p}_{L}}
\lambda_{k-1}\lambda_{k+1}\left(\frac{\cos\theta_{k+1}}{\sin\theta_{k}}\bm{q}_{h}\cdot\bm{n}_{k+1}
-\frac{\cos\theta_{k-1}}{\sin\theta_{k}}\bm{q}_{h}\cdot\bm{n}_{k-1}\right)\\
&=\sum_{k=1}^{3}\left\{\int_{e_{k-1}}\frac{\ell_{k-1}^{2}}{2}\bm{n}_{k-1}\cdot\dftkm{\bm{p}_{L}}
\lambda_{k}\lambda_{k+1}\frac{\cos\theta_{k}}{\sin\theta_{k-1}}\right.\\
&\quad\quad\quad\quad\quad\quad\quad\left.-\int_{e_{k+1}}\frac{\ell_{k+1}^{2}}{2}\bm{n}_{k+1}\cdot\dftkp{\bm{p}_{L}}
\lambda_{k}\lambda_{k-1}\frac{\cos\theta_{k}}{\sin\theta_{k+1}}\right\}\bm{q}_{h}\cdot\bm{n}_{k}\\
&=\sum_{k=1}^{3}\cot\theta_{k}\int_{e_{k}}\lambda_{k+1}\lambda_{k-1}
\left(\frac{\ell_{k-1}^{2}}{2}\bm{n}_{k-1}\cdot\dftkm{\bm{p}_{L}}-
\frac{\ell_{k+1}^{2}}{2}\bm{n}_{k+1}\cdot\dftkp{\bm{p}_{L}}\right)\bm{q}_{h}\cdot\bm{n}_{k}.
\end{aligned}
\end{equation*}
Then by \cref{decomp2,decomp3} and following identities:
\begin{equation*}
\begin{aligned}
&\ell_{k-1}\sin\theta_{k+1}=\ell_{k+1}\sin\theta_{k-1}=d_{k},\\
&\ell_{k-1}^{2}\cos^{2}\theta_{k+1}-\ell_{k+1}^{2}\cos^{2}\theta_{k-1}=\ell_{k-1}^{2}-\ell_{k+1}^{2},\\
&\ell_{k+1}\cos\theta_{k-1}+\ell_{k-1}\cos\theta_{k+1}=\ell_{k},
\end{aligned}
\end{equation*}
and direct calculation, we obtain \cref{locallemmaeq}.
\end{proof}

Now we state definitions of $\mathcal{O}(h^{1+\alpha})$ approximate parallelograms and mildly 
structured grids in \cite{BaXu2003} below with a little generalization.
\begin{definition}
Let e be an edge in the triangulation $\mathcal{T}_{h}$. Let $\tau$ and $\tau^{\prime}$ be the two adjacent
elements sharing e. We say that $\tau$ and $\tau^{\prime}$ form an $\mathcal{O}(h^{1+\alpha})$ approximate parallelogram
if the lengths of any two opposite edges differ only by $\mathcal{O}(h^{1+\alpha})$.
\end{definition}
The boundary elements need more delicate treatment.
\begin{definition}\label{approxpara}
Let x be a vertex in $\mathcal{T}_{h}$ on $\partial\Omega$. Let e and $e^{\prime}$ be the two boundary edges
sharing x as an endpoint, and let $\bm{t}$ and $\bm{t}^{\prime}$ be the unit tangents, oriented counterclockwise.
Let $\tau$ and $\tau^{\prime}$ be the two adjacent elements having e and $e^{\prime}$ as edges respectively.
Number e and $e^{\prime}$ as a pair of corresponding edges. By going along the boundaries of $\tau$ and $\tau^{\prime}$
counterclockwise, we have other two pairs of corresponding edges. We say that $\tau$ and $\tau^{\prime}$ form an $\mathcal{O}(h^{1+\alpha})$ approximate parallelogram if the lengths of any two corresponding edges differ only by $\mathcal{O}(h^{1+\alpha})$,
and $|\bm{t}-\bm{t}^{\prime}|=\mathcal{O}(h^{\alpha})$.
\end{definition}
\begin{remark}
$\tau$ and $\tau^{\prime}$ in \cref{approxpara} don't form an approximate parallelogram in the \emph{usual} sense,
since they have no common edge.
\end{remark}
\begin{definition}\label{alphasigma}
The triangulation $\mathcal{T}_{h}$ satisfies the $(\alpha,\sigma)$-condition if the following hold:
\begin{enumerate}
\item Let $\mathcal{E}=\mathcal{E}_{1}\biguplus\mathcal{E}_{2}$ be the set of interior edges. For each $e\in\mathcal{E}_{1}$,
$\tau$ and $\tau^{\prime}$ form an $\mathcal{O}(h^{1+\alpha})$ approximate parallelogram, while
$\sum_{e\in\mathcal{E}_{2}}|\tau|+|\tau^{\prime}|=\mathcal{O}(h^{\sigma}).$
\item Let $\mathcal{P}=\mathcal{P}_{1}\biguplus\mathcal{P}_{2}$ be the set of boundary vertices. The adjacent boundary elements $\tau, \tau^{\prime}$ in \cref{approxpara} associated with each $x\in\mathcal{P}_{1}$ form an $\mathcal{O}(h^{1+\alpha})$ approximate parallelogram, and $|\mathcal{P}_{2}|=\kappa$ is a finite number independent of $h$.
\label{alphasigma:P}
\end{enumerate}
\end{definition}
For example, we have $\alpha=\sigma=\infty, \mathcal{E}_{2}=\emptyset, \kappa=4$ for the uniform grid in \cref{grida}.
\begin{definition}\label{defpw}
Let $\Omega$ be decomposed into $N$ subdomains, where $N$ is independent of $h$.  $\mathcal{T}_{h}$ is said to satisfy the piecewise
$(\alpha,\sigma)$-condition if the restriction of $\mathcal{T}_{h}$ to each subdomain satisfies the $(\alpha,\sigma)$-condition.
\end{definition}

With the above definitions, we can present the main lemma.
\begin{lemma}\label{mainlemma}
Let $\mathcal{T}_{h}$ be quasi-uniform and satisfy the $(\alpha,\sigma)$-condition.
Let $\bm{q}_{h}\in\curlvect\mathcal{S}_{h}$, where $\mathcal{S}_{h}$ consists of continuous piecewise linear polynomials
on $\mathcal{T}_{h}$. Then
\begin{equation}\label{mainlemmaeq}
|(\bm{p}-\Pi_{h}\bm{p},\bm{q}_{h})|
\lesssim h^{1+\rho}|\log h|^{\frac{1}{2}}||\nabla\bm{p}||_{1,\infty,\Omega}||\bm{q}_{h}||_{0,\Omega},
\end{equation}
where
\begin{equation*}
\rho=\min(1,\alpha,\frac{\sigma}{2}).
\end{equation*}
\end{lemma}
\begin{proof}
By \cref{interp:cr:rt,err1} and passing through $\I_{h}^{CR}\bm{p}$, we have
\begin{equation}\label{maintotal}
\begin{aligned}
&\quad(\bm{p}-\Pi_{h}\bm{p},\bm{q}_{h})\\
&=(\bm{p}-\I_{h}^{CR}\bm{p},\bm{q}_{h})
+\sum_{\tau\in\mathcal{T}_{h}}\int_{\tau}(\I_{h}^{CR}\bm{p}-\Pi_{h}\I_{h}^{CR}\bm{p})\cdot\bm{q}_{h}\\
&=(\bm{p}-\I_{h}^{CR}\bm{p},\bm{q}_{h})\\
&+\sum_{\tau\in\mathcal{T}_{h}}\sum_{k=1}^{3}
\cot\theta_{k}\int_{e_{k}}\lambda_{k-1}\lambda_{k+1}\left(\sum_{j=1}^{3}
\alpha_{k}^{(j)}\mathcal{A}_{k}^{(j)}(\I_{h}^{CR}\bm{p}-\bm{p})\right)\bm{q}_{h}\cdot\bm{n}_{k}\\
&+\sum_{\tau\in\mathcal{T}_{h}}\sum_{k=1}^{3}
\cot\theta_{k}\int_{e_{k}}\lambda_{k-1}\lambda_{k+1}\left(\sum_{j=1}^{3}
\alpha_{k}^{(j)}\mathcal{A}_{k}^{(j)}\bm{p}\right)\bm{q}_{h}\cdot\bm{n}_{k}\\
&=\Rom{1}+\Rom{2}+\Rom{3}.\\
\end{aligned}
\end{equation}
Parts $\Rom{1}$ and $\Rom{2}$ can be simply estimated by the standard finite element interpolation theory:
\begin{equation}\label{bdI}
|\Rom{1}|\lesssim h^{2}|\bm{p}|_{2,\Omega}||\bm{q}_{h}||_{0,\Omega},
\end{equation}
and
\begin{equation}\label{bdII}
\begin{aligned}
|\Rom{2}|&\lesssim\sum_{\tau\in\mathcal{T}_{h}}h\int_{\tau}|\nabla(\I_{h}^{CR}\bm{p}-\bm{p})|\cdot|\bm{q}_{h}|+h^{2}\int_{\tau}|\nabla^{2}\bm{p}|\cdot|\bm{q}_{h}|\\
&\lesssim h^{2}|\bm{p}|_{2,\Omega}||\bm{q}_{h}||_{0,\Omega}.
\end{aligned}
\end{equation}
The main task is to bound part $\Rom{3}$. For $e\subset\partial\Omega$, let $\tau$ be the element having $e$ as an edge. For 
$e\in\mathcal{E}$, let $\tau$ and $\tau^{\prime}$ be the two elements sharing $e$.
Let $\bm{t}_{e}$ and $\bm{n}_{e}$ denote the unit tangent and normal to $e$ whose directions are consistent with $\tau$.
Let $\mathcal{A}_{e}^{(j)}$ denote the operators in \cref{operatorA}
corresponding to $\bm{t}_{e}$ and $\bm{n}_{e}$, $\theta_{e}$ the angle opposite to $e$ in $\tau$, $\ell_{e}$ the length of $e$,
$\alpha_{e}^{(j)}$ the quantity associated with $e$ on $\tau$ in \cref{alpha}.
Corresponding quantities on $\tau^{\prime}$ have a superscript $\prime$.
Denote
\begin{equation*}
b_{e}=\lambda_{k-1}\lambda_{k+1}, \quad
\beta_{e}^{(j)}=\alpha_{e}^{(j)}\cot\theta_{e}-\alpha_{e}^{(j)\prime}\cot\theta_{e}^{\prime}.
\end{equation*}
$\bm{q}_{h}\in\mathcal{Q}_{h}$  implies that 
$\bm{q}_{h}|_{\tau}\cdot\bm{n}_{e}=\bm{q}_{h}|_{\tau^{\prime}}\cdot\bm{n}_{e}$ on $e$. Thus we can transform $\Rom{3}$
from element-wise summation to edge-wise summation:
\begin{equation*}
\Rom{3}=\Rom{3}_{1}+\Rom{3}_{2}+\Rom{3}_{3},
\end{equation*}
where
\begin{equation*}
\begin{aligned}
&\Rom{3}_{i}=\sum_{e\in\mathcal{E}_{i}}\int_{e}b_{e}\left(\sum_{j=1}^{3}
\beta_{e}^{(j)}\mathcal{A}_{e}^{(j)}\bm{p}\right)\bm{q}_{h}\cdot\bm{n}_{e},\quad i=1,2,\\
&\Rom{3}_{3}=\sum_{e\subset\partial\Omega}\cot\theta_{e}\int_{e}b_{e}\left(\sum_{j=1}^{3}
\alpha_{e}^{(j)}\mathcal{A}_{e}^{(j)}\bm{p}\right)\bm{q}_{h}\cdot\bm{n}_{e}.
\end{aligned}
\end{equation*}
For $e\in\mathcal{E}_{1}$, the fact that $\tau$ and $\tau^{\prime}$ form an $\mathcal{O}(h^{1+\alpha})$ 
approximate parallelogram implies $$|\beta_{e}^{(j)}|\lesssim h^{2+\alpha}.$$
Combining the above inequality
with the trace inequality
\begin{equation}\label{trace}
\int_{\partial\tau}|f|\lesssim h^{-1}\int_{\tau}|f|+\int_{\tau}|\nabla f|
\end{equation}
leads to
\begin{equation}\label{bdIII1}
\begin{aligned}
|\Rom{3}_{1}|&\lesssim\sum_{e\in\mathcal{E}_{1}}h^{2+\alpha}\left\{h^{-1}\int_{\tau}|\nabla\bm{p}|\cdot|\bm{q}_{h}|
+\int_{\tau}|\nabla^{2}\bm{p}|\cdot|\bm{q}_{h}|\right\}\\
&\lesssim h^{1+\alpha}\sum_{e\in\mathcal{E}_{1}}\int_{\tau}(|\nabla\bm{p}|+|\nabla^{2}\bm{p}|)\cdot|\bm{q}_{h}|\\
&\lesssim h^{1+\alpha}||\nabla\bm{p}||_{1,\Omega}||\bm{q}_{h}||_{0,\Omega}.
\end{aligned}
\end{equation}
For $e\in\mathcal{E}_{2}$, there is no cancellation. Due to the small measure of the region covered by elements near $e\in\mathcal{E}_{2}$, 
$\Rom{3}_{2}$ is estimated by
\begin{equation}\label{bdIII2}
\begin{aligned}
|\Rom{3}_{2}|&\lesssim h^{3}\sum_{e\in\mathcal{E}_{2}}||\nabla\bm{p}||_{0,\infty,\tau}||\bm{q}_{h}||_{0,\infty,\tau}\\
&\lesssim h||\nabla\bm{p}||_{0,\infty,\Omega}\sum_{e\in\mathcal{E}_{2}}\int_{\tau}|\bm{q}_{h}|\\
&\lesssim h^{1+\frac{\sigma}{2}}||\nabla\bm{p}||_{0,\infty,\Omega}||\bm{q}_{h}||_{0,\Omega}.
\end{aligned}
\end{equation}
The trickiest part of this proof is to bound $\Rom{3}_{3}$. Let $\bm{q}_{h}=\curlvect w_{h}$, where $w_{h}\in\mathcal{S}_{h}$.
We can assume that $\int_{\Omega}w_{h}=0$ by subtracting a constant from $w_{h}$. 
Then by the Poincar\'{e} inequality, we have
\begin{equation}\label{bdqw}
||w_{h}||_{1,\Omega}\lesssim||\bm{q}_{h}||_{0,\Omega}.
\end{equation}
Denote
\begin{equation*}
B_{e}^{(j)}=\alpha_{e}^{(j)}\mathcal{A}_{e}^{(j)}\bm{p}\cot\theta_{e},\quad\bar{B}_{e}^{(j)}=\frac{1}{\ell_{e}}\int_{e}B_{e}.
\end{equation*}
By \cref{rottk,curl:grad}, 
$$\bm{q}_{h}\cdot\bm{n}_{e}=\dfte{w_{h}}.$$ 
Then part $\Rom{3}_{3}$ can be split into:
\begin{equation*}
\begin{aligned}
\Rom{3}_{3}&=\sum_{e\subset\partial\Omega}\int_{e}b_{e}\sum_{j=1}^{3}
(B_{e}^{(j)}-\bar{B}_{e}^{(j)})\dfte{w_{h}}\\
&+\sum_{e\subset\partial\Omega}\int_{e}b_{e}\sum_{j=1}^{3}
\bar{B}_{e}^{(j)}\dfte{w_{h}}=\Rom{3}_{3}^{(1)}+\Rom{3}_{3}^{(2)}.
\end{aligned}
\end{equation*}
The first term can be estimated by \cref{bdqw}:
\begin{equation}\label{bdIII31}
\begin{aligned}
|\Rom{3}_{3}^{(1)}|&\lesssim h^{3}|\bm{p}|_{2,\infty,\Omega}\sum_{e\subset\partial\Omega}\int_{e}|\nabla w_{h}|\\
&\lesssim h^{2}|\bm{p}|_{2,\infty,\Omega}\sum_{e\subset\partial\Omega}\int_{\tau}|\nabla w_{h}|\\
&\lesssim h^{2}|\bm{p}|_{2,\infty,\Omega}||\bm{q}_{h}||_{0,\Omega}.
\end{aligned}
\end{equation}
For $x\in\mathcal{P}$, let $e, e^{\prime}$ be the two edges on $\partial\Omega$ sharing $x$ as an ending point. 
Then the second term becomes
\begin{equation*}
\begin{aligned}
\Rom{3}_{3}^{(2)}&=\sum_{e\subset\partial\Omega}\frac{\ell_{e}}{6}\sum_{j=1}^{3}
\bar{B}_{e}^{(j)}\dfte{w_{h}}\\
&=\sum_{x\in\mathcal{P}}\frac{1}{6}\sum_{j=1}^{3}
\left(\bar{B}_{e}^{(j)}-\bar{B}_{e^{\prime}}^{(j)}\right)w_{h}(x).
\end{aligned}
\end{equation*}
For $x\in\mathcal{P}_{1}$, definitions \cref{alpha,operatorA} together with the $(\alpha,\sigma)$-condition along the boundary
imply cancellation and thus 
\begin{equation}\label{Bep1}
\left|\bar{B}_{e}^{(j)}-\bar{B}_{e^{\prime}}^{(j)}\right|\lesssim h^{2+\alpha}||\nabla\bm{p}||_{1,\infty,\Omega}.
\end{equation}
For $x\in\mathcal{P}_{2}$,
\begin{equation}\label{Bep2}
\left|\bar{B}_{e}^{(j)}-\bar{B}_{e^{\prime}}^{(j)}\right|\leq
\left|\bar{B}_{e}^{(j)}\right|+\left|\bar{B}_{e^{\prime}}^{(j)}\right|\lesssim h^{2}||\nabla\bm{p}||_{0,\infty,\Omega}.
\end{equation}
It follows from the discrete Sobolev inequality
\begin{equation}\label{disSobolev}
||w_{h}||_{0,\infty,\Omega}\lesssim|\log h|^{\frac{1}{2}}||w_{h}||_{1,\Omega},
\end{equation}
the quasi-uniformity, \cref{bdqw,Bep1,Bep2} that
\begin{equation}\label{bdIII32}
\begin{aligned}
|\Rom{3}_{3}^{(2)}|&\lesssim\left(\sum_{x\in\mathcal{P}_{1}}
h^{2+\alpha}||\nabla\bm{p}||_{1,\infty,\Omega}+
\sum_{x\in\mathcal{P}_{2}}h^{2}||\nabla\bm{p}||_{0,\infty,\Omega}\right)||w||_{0,\infty,\partial\Omega}\\
&\lesssim h^{1+\alpha}|\log h|^{\frac{1}{2}}||\nabla\bm{p}||_{1,\infty,\Omega}||w_{h}||_{1,\Omega}\\
&\lesssim h^{1+\alpha}|\log h|^{\frac{1}{2}}||\nabla\bm{p}||_{1,\infty,\Omega}||\bm{q}_{h}||_{0,\Omega}.
\end{aligned}
\end{equation}
Now, combining \cref{maintotal,bdI,bdII,bdIII1,bdIII2,bdIII31,bdIII32}, we obtain \cref{mainlemma}.
\end{proof}

In the rest of this paper, $\rho$ \emph{always} denotes $\min(1,\alpha,\sigma/2)$ unless confusion arises.

\begin{remark}\label{piecewise}
The proof of \cref{mainlemma} is completely local and thus \cref{mainlemmaeq} holds on piecewise $(\alpha,\sigma)$-grids.
\end{remark}

\cref{mainlemma} can be easily generalized. By checking the proof of \cref{mainlemma}, one can see that the quasi-uniformity is only used to guarantee the discrete Sobolev inequality and bound the number of vertices lying on $\partial\Omega$. By test function $\bm{q}_{h}\in\curlvect{S}_{h}$ for the Neumann boundary condition having zero normal trace,
we have the following estimate.
\begin{corollary}\label{lemmaNeumann}
Let $\mathcal{T}_{h}$ be an $(\alpha,\sigma)$-grid without assuming adjacent boundary elements form approximate parallelograms
in \cref{alphasigma}. Then
for $\bm{q}_{h}\in\curlvect\mathcal{S}_{h}$ and $\bm{q}_{h}\cdot\bm{n}=0$,
\begin{equation*}
|(\bm{p}-\Pi_{h}\bm{p},\bm{q}_{h})|\lesssim h^{1+\rho}
\left(||\nabla\bm{p}||_{0,\infty,\Omega}+||\bm{p}||_{2,\Omega}\right)||\bm{q}_{h}||_{0,\Omega}.
\end{equation*}
\end{corollary}
\begin{proof}
The proof is basically the same as \cref{mainlemma}. But we get rid of bounding $\Rom{3}_{3}$, the error occuring on boundary elements.
Then the regularity assumption is weakened.
\end{proof}

Let $\text{diam}\tau$ denote the diameter of $\tau$. The quasi-uniformity in \cref{mainlemma} can be replaced by
\begin{equation}\label{hgamma}
\min_{\tau\in\mathcal{T}_{h}}\text{diam}\tau\gtrsim h^{\gamma},\quad\gamma\geq1.
\end{equation}
\begin{corollary}\label{lemmahgamma}
Assume the condition \cref{hgamma} instead of quasi-uniformity in \cref{mainlemma}. Then
\cref{mainlemmaeq} holds with a smaller $\rho$:
\begin{equation*}
\rho=\min(1,1+\alpha-\gamma,\frac{\sigma}{2}).
\end{equation*}
\end{corollary}
\begin{proof}
The proof is basically the same as \cref{mainlemma}. \cref{hgamma} ensures the discrete Sobolev inequality (cf. \cite{Brenner2008}).
The number of boundary vertices is bounded by
\begin{equation*}
|\mathcal{P}|\lesssim h^{-\gamma}.
\end{equation*}
Therefore, the only difference is that the bound for $\Rom{3}_{3}^{(2)}$ becomes
\begin{equation*}
|\Rom{3}_{3}^{(2)}|\lesssim h^{1+\min(1+\alpha-\gamma,1)}|\log h|^{\frac{1}{2}}
||\nabla\bm{p}||_{1,\infty,\Omega}||\bm{q}_{h}||_{0,\Omega}.
\end{equation*}
\end{proof}
\begin{remark}
The condition \cref{hgamma} also appears in pointwise a posteriori error estimation (cf. \cite{Ainsworth2000}).
\end{remark}

At the end of this section, we show that the logarithmic factor in \cref{mainlemma} can be removed by using the following lemma
proved by Brandts \cite{Brandts1994}.
\begin{lemma}\label{bdynorm}
Let $\Omega_{h}=\{x\in\Omega: \dist(x,\partial\Omega)\leq h\}$. Then
\begin{equation*}
||u||_{0,\Omega_{h}}\lesssim h^{s}||u||_{s,\Omega},\quad0\leq s\leq\frac{1}{2}.
\end{equation*}
\end{lemma}
\begin{corollary}\label{lemmanolog}
Assume the same conditions in \cref{mainlemma}. Then $\forall\varepsilon>0$,
\begin{equation*}
|(\bm{p}-\Pi_{h}\bm{p},\bm{q}_{h})|
\lesssim h^{1+\rho}||\nabla\bm{p}||_{2+\varepsilon,\Omega}||\bm{q}_{h}||_{0,\Omega},
\end{equation*}
\end{corollary}
\begin{proof}
The bounds for $\Rom{1}, \Rom{2}, \Rom{3}_{1}, \Rom{3}_{2}, \Rom{3}_{3}^{(1)}$ are the same as \cref{mainlemma}.
The bound for $\Rom{3}_{3}^{(2)}$ is improved by the Sobolev embedding
$H^{2+\frac{\varepsilon}{2}}(\Omega_{h})\subset W^{1}_{\infty}(\Omega_{h})$ and \cref{bdynorm}:
\begin{equation*}
\begin{aligned}
|\Rom{3}_{3}^{(2)}|&\lesssim\left(\sum_{x\in\mathcal{P}_{1}}
h^{2+\alpha}||\nabla\bm{p}||_{1,\infty,\Omega_{h}}+
\sum_{x\in\mathcal{P}_{2}}h^{2}||\nabla\bm{p}||_{0,\infty,\Omega_{h}}\right)||w||_{0,\infty,\partial\Omega}\\
&\lesssim h^{1+\alpha}|\log h|^{\frac{1}{2}}||\nabla\bm{p}||_{1,\infty,\Omega_{h}}||w_{h}||_{1,\Omega}\\
&\lesssim h^{1+\alpha}|\log h|^{\frac{1}{2}}||\nabla\bm{p}||_{2+\frac{\varepsilon}{2},\Omega_{h}}||\bm{q}_{h}||_{0,\Omega}\\
&\lesssim h^{1+\alpha+\frac{\varepsilon}{2}}|\log h|^{\frac{1}{2}}||\nabla\bm{p}||_{2+\varepsilon,\Omega}||\bm{q}_{h}||_{0,\Omega}.
\end{aligned}
\end{equation*}
Therefore, $|\log h|^{1/2}$ is compensated by $h^{\varepsilon/2}$.
\end{proof}
\begin{remark}\label{improveBaXu}
The factor $|\log h|^{1/2}$ appearing in Lemma 2.5 and Theorem 3.1 in \cite{BaXu2003} can be removed in the same way.
\end{remark}

\section{Superconvergence results}\label{sec:superp}
From \cref{mix:v,mix:dv}, we have the error equation
\begin{subequations}\label{err:eqn}
\begin{align}
&(\bm{\alpha}(\bm{p}-\bm{p}_{h}),\bm{q}_{h})-(\bm{q}_{h},\bm{\beta}(u-u_{h}))
+(\divg\bm{q}_{h},u-u_{h})=0,\quad\bm{q}_{h} \in\mathcal{Q}_{h},\\
&-(\divg(\bm{p}-\bm{p}_{h}),v_{h})+(c(u-u_{h}),v_{h})=0,\quad v_{h}\in\mathcal{V}_{h}.
\end{align}
\end{subequations}

From \cite{Douglas1985}, we have the following a priori error estimates:
\begin{subequations}\label{apriori}
\begin{align}
&||\bm{p}-\bm{p}_{h}||_{0,\Omega}\lesssim h||u||_{2,\Omega},\\
&||\divg(\bm{p}-\bm{p}_{h})||_{0,\Omega}\lesssim h||u||_{3,\Omega},\\
&||u-u_{h}||_{0,\Omega}\lesssim h||u||_{2,\Omega}.
\end{align}
\end{subequations}

The following is the well-known superconvergence result for $||P_{h}u-u_{h}||_{0,\Omega}$ on general unstructured meshes
(cf. \cite{Douglas1985}). 
\begin{theorem}\label{supercloseu}
\begin{equation*}
||P_{h}u-u_{h}||_{0,\Omega}\lesssim h^{2}||u||_{3,\Omega}.
\end{equation*}
\end{theorem}

Then we prove the superconvergence for $||\divg(\Pi_{h}\bm{p}-\bm{p}_{h})||_{0,\Omega}$.
\begin{theorem}\label{superclosedivp}
\begin{equation*}
||\divg(\Pi_{h}\bm{p}-\bm{p}_{h})||_{0,\Omega}\lesssim h^{2}||u||_{3,\Omega}.
\end{equation*}
\end{theorem}
\begin{proof}
From \cref{com:diag,approxRT,supercloseu,apriori,err:eqn}, it follows that for $v_{h}\in\mathcal{V}_{h}$,
\begin{equation}\label{proofsuperclosedivp}
\begin{aligned}
(\divg(\Pi_{h}\bm{p}-\bm{p}_{h}),v_{h})&=(P_{h}\divg\bm{p}-\divg\bm{p}_{h},v_{h})\\
&=(\divg(\bm{p}-\bm{p}_{h}),v_{h})\\
&=(u-P_{h}u,cv_{h})+(P_{h}u-u_{h},cv_{h})\\
&=(u-P_{h}u,cv_{h}-P_{h}(cv_{h}))+\mathcal{O}(h^{2})||u||_{3,\Omega}||v_{h}||_{0,\Omega}\\
&=\mathcal{O}(h^{2})||u||_{3,\Omega}||v_{h}||_{0,\Omega}.
\end{aligned}
\end{equation}
Therefore, \cref{superclosedivp} follows from setting $v_{h}=\divg(\Pi_{h}\bm{p}-\bm{p}_{h})$ in \cref{proofsuperclosedivp}.
\end{proof}

Before proving the superconvergence result for $||\Pi_{h}\bm{p}-\bm{p}_{h}||_{0,\Omega}$, it is necessary to
state two lemmas. The first lemma is due to Raviart and Thomas \cite{RT1977}.
\begin{lemma}\label{invdiv}
For $v_{h}\in\mathcal{V}_{h}$, there exists $\bm{q}_{h}\in\mathcal{Q}_{h}$, such that
\begin{equation*}
\divg\bm{q}_{h}=v_{h},\quad||\bm{q}_{h}||_{0,\Omega}\lesssim||v_{h}||_{0,\Omega}.
\end{equation*}
\end{lemma}

Another useful lemma is the discrete Helmholtz decomposition (cf. \cite{Chen2008} and references therein).
\begin{lemma}\label{disHelmholtz}
$\mathcal{Q}_{h}$ has the following orthogonal decomposition with respect to $(\cdot,\cdot)$:
\begin{equation*}
\mathcal{Q}_{h}=\grad_{h}\mathcal{V}_{h}\oplus\overrightarrow{\nabla}\times\mathcal{S}_{h},
\end{equation*}
where $\grad_{h}: \mathcal{V}_{h}\to\mathcal{Q}_{h}^{*}$ is defined by
\begin{equation*}
(\grad_{h}v_{h},\bm{q}_{h})=-(v_{h},\divg\bm{q}_{h}),\quad\bm{q}_{h}\in\mathcal{Q}_{h}.
\end{equation*}
\end{lemma}

The following is a result from
\cref{invdiv,disHelmholtz,lemmanolog,supercloseu,superclosedivp}.
\begin{theorem}\label{superclosep}
Let $\mathcal{T}_{h}$ be a quasi-uniform and piecewise $(\alpha,\sigma)$-grid. Then
$\forall\varepsilon>0$,
\begin{equation*}
||\Pi_{h}\bm{p}-\bm{p}_{h}||_{0,\Omega}\lesssim
h^{1+\rho}||u||_{4+\varepsilon,\Omega}.
\end{equation*}
\end{theorem}
\begin{proof}
Let $\bm{\xi}_{h}=\Pi_{h}\bm{p}-\bm{p}_{h}$. \cref{disHelmholtz} gives
\begin{equation}\label{disHelmholtzdecomp}
\bm{\xi}_{h}=\grad_{h}v_{h}\oplus\curlvect w_{h},
\end{equation}
where $(v_{h},w_{h})\in\mathcal{V}_{h}\times\mathcal{S}_{h}$, and
\begin{subequations}\label{bdvw}
\begin{align}
&||\grad_{h}v_{h}||_{0,\Omega}\lesssim||\bm{\xi}_{h}||_{0,\Omega},\\
&||w_{h}||_{0,\Omega}\lesssim||\curlvect w_{h}||_{0,\Omega}\lesssim||\bm{\xi}_{h}||_{0,\Omega}.
\end{align}
\end{subequations}
Let $\tilde{\bm{q}}_{h}\in\mathcal{Q}_{h}$ be the preimage of $v_{h}$ under $\divg$ in \cref{invdiv}. Then
\begin{equation*}
\begin{aligned}
||v_{h}||_{0,\Omega}^{2}&=-(\grad_{h}v_{h},\tilde{\bm{q}}_{h})\\
&\lesssim||\grad_{h}v_{h}||_{0,\Omega}||v_{h}||_{0,\Omega},
\end{aligned}
\end{equation*}
and thus
\begin{equation}\label{disPoin}
||v_{h}||_{0,\Omega}\lesssim||\grad_{h}v_{h}||_{0,\Omega}.
\end{equation}
By \cref{disPoin,disHelmholtzdecomp}, we have
\begin{equation*}
\begin{aligned}
||\grad_{h}v_{h}||_{0,\Omega}^{2}&=-(v_{h},\divg\grad_{h}v_{h})\\
&=-(v_{h},\divg\bm{\xi}_{h})\\
&\lesssim||\grad_{h}v_{h}||_{0,\Omega}||\divg\bm{\xi}_{h}||_{0,\Omega}.
\end{aligned}
\end{equation*}
Then it follows from \cref{superclosedivp} that
\begin{equation}\label{bdgradv}
||\grad_{h}v_{h}||_{0,\Omega}\lesssim h^{2}||u||_{3,\Omega}.
\end{equation}
It remains to bound $\curlvect{w_{h}}$. By the orthogonality in \cref{disHelmholtzdecomp},
\begin{equation}\label{curlwtotal}
\begin{aligned}
||\curlvect{w_{h}}||_{0,\Omega}^{2}
&=(\Pi_{h}\bm{p}-\bm{p}_{h},\curlvect w_{h})\\
&=-(\bm{p}-\Pi_{h}\bm{p},\curlvect w_{h})\\
&+(\bm{\alpha}(\bm{p}-\bm{p}_{h}),\bm{A}\curlvect w_{h}-\Pi_{h}\bm{A}\curlvect w_{h})\\
&+(\bm{\alpha}(\bm{p}-\bm{p}_{h}),\Pi_{h}\bm{A}\curlvect w_{h})\\
&=\Rom{1}+\Rom{2}+\Rom{3}.
\end{aligned}
\end{equation}
$\Rom{1}$ is estimated by \cref{lemmanolog}
\begin{equation}\label{bd1curlw}
|\Rom{1}|\lesssim h^{1+\rho}||\nabla\bm{p}||_{2+\varepsilon,\Omega}||\curlvect w_{h}||_{0,\Omega}.
\end{equation}
$\Rom{2}$ is estimated by \cref{approxRT,apriori}
\begin{equation}\label{bd2curlw}
|\Rom{2}|\lesssim h^{2}||u||_{2,\Omega}||\curlvect w_{h}||_{0,\Omega}.
\end{equation}
As for $\Rom{3}$, setting $\bm{q}_{h}=\Pi_{h}\bm{A}\curlvect w_{h}$ in \cref{err:eqn} leads to
\begin{equation*}
\begin{aligned}
\Rom{3}&=(\bm{q}_{h},\bm{\beta}(u-u_{h}))-(\divg\bm{q}_{h},u-u_{h})\\
&=\Rom{3}_{1}+\Rom{3}_{2}.\\
\end{aligned}
\end{equation*}
By \cref{approxRTa,com:diag,formRT}, we have
\begin{equation}\label{stableinterp}
||\bm{q}_{h}||_{0,\Omega}\lesssim||\curlvect w_{h}||_{0,\Omega},
\end{equation}
and
\begin{equation}\label{stableinterp2}
||\nabla_{h}\bm{q}_{h}||_{0,\Omega}\lesssim||\divg\bm{q}_{h}||_{0,\Omega}\lesssim||\curlvect w_{h}||_{0,\Omega}.
\end{equation}
Then $\Rom{3}_{1}$ can be estimated by \cref{apriori,supercloseu,stableinterp,stableinterp2}:
\begin{equation}\label{bd31curlw}
\begin{aligned}
\Rom{3}_{1}&=(\bm{\beta}\cdot\bm{q}_{h},u-P_{h}u+P_{h}u-u_{h})\\
&=(\bm{\beta}\cdot\bm{q}_{h}-P_{h}\bm{\beta}\cdot\bm{q}_{h},u-P_{h}u)+\mathcal{O}(h^{2})
||u||_{3,\Omega}||\bm{q}_{h}||_{0,\Omega}\\
&=\mathcal{O}(h^{2})||\nabla_{h}(\bm{\beta}\cdot\bm{q}_{h})||_{0,\Omega}||u||_{1,\Omega}
+\mathcal{O}(h^{2})||u||_{3,\Omega}||\bm{q}_{h}||_{0,\Omega}\\
&=\mathcal{O}(h^{2})||u||_{3,\Omega}||\curlvect{w_{h}}||_{0,\Omega}.
\end{aligned}
\end{equation}
$\Rom{3}_{2}$ can be estimated by \cref{superclosedivp,stableinterp2}:
\begin{equation}\label{bd32curlw}
\begin{aligned}
\Rom{3}_{2}&=(\divg\bm{q}_{h},P_{h}u-u_{h})\\
&=\mathcal{O}(h^{2})||u||_{3,\Omega}||\curlvect w_{h}||_{0,\Omega}.
\end{aligned}
\end{equation}
Combining \cref{curlwtotal,bd1curlw,bd2curlw,bd31curlw,bd32curlw}, we obtain
\begin{equation}\label{bdcurlw}
||\curlvect{w_{h}}||_{0,\Omega}\lesssim h^{1+\rho}||u||_{4+\varepsilon,\Omega}.
\end{equation}
Then \cref{superclosep} follows from \cref{disHelmholtzdecomp,bdgradv,bdcurlw}.
\end{proof}

\begin{remark}\label{bc=0}
If $\bm{b}=\bm{0}, c=0$ in \cref{mix:c1}, then $\divg(\Pi_{h}\bm{p}-\bm{p}_{h})=0$, which implies that
$\Pi_{h}\bm{p}-\bm{p}_{h}$ is a piecewise constant function. Then the superconvergence analysis
in this section simplifies. In particular, to prove \cref{superclosep}, it is not necessary to employ \cref{disHelmholtz,supercloseu,superclosedivp} 
in this simplified case.
\end{remark}

\cref{superclosep} can be easily generalized by \cref{mainlemma,lemmaNeumann,lemmahgamma} respectively. 
Here we present a theorem from \cref{lemmaNeumann}.
\begin{theorem}\label{superclosepNeumann}
Let $\mathcal{T}_{h}$ be an $(\alpha,\sigma)$-grid without assuming adjacent boundary elements
form approximate parallelograms in \cref{alphasigma}. Then under the Neumann boundary condition
$\bm{p}\cdot\bm{n}=g$ on $\partial\Omega$,
\begin{equation*}
||\Pi_{h}\bm{p}-\bm{p}_{h}||_{0,\Omega}\lesssim
h^{1+\rho}(||\nabla\bm{p}||_{0,\infty,\Omega}+||u||_{3,\Omega}).
\end{equation*}
\end{theorem}
\begin{proof}
The proof is basically the same as \cref{superclosep}. The difference is from function spaces. 
For the Neumann boundary condition, the test function spaces in \cref{mix:va,mix:dva} are
$\mathcal{Q}_{0}=\{\bm{q}\in\mathcal{Q}: \bm{q}\cdot\bm{n}=0\}$ and
$\mathcal{Q}_{0h}=\{\bm{q}\in\mathcal{Q}_{h}: \bm{q}_{h}\cdot\bm{n}=0\}$, respectively.
The numerical solution $\bm{p}_{h}$ is in $\mathcal{Q}_{h}$ with the constraint $\bm{p}_{h}\cdot\bm{n}=g_{h}$, where
\begin{equation*}
g_{h}|_{e}=\frac{1}{\ell_{e}}\int_{e}g,\quad e\subset\partial\Omega.
\end{equation*}
Thus
\begin{equation}\label{NeumannI}
\int_{e}(\Pi_{h}\bm{p}-\bm{p}_{h})\cdot\bm{n}_{e}=0.
\end{equation}
The form of $\mathcal{RT}_{0}(\tau)$ \cref{formRT} implies
\begin{equation}\label{NeumannII}
\left(\dfte{\Pi_{h}\bm{p}}-\dfte{\bm{p}_{h}}\right)\cdot\bm{n}_{e}=0.
\end{equation}
Combining \cref{NeumannI,NeumannII}, we have $\bm{\xi}_{h}=\Pi_{h}\bm{p}-\bm{p}_{h}\in\mathcal{Q}_{0h}$.
By the discrete Helmholtz decomposition for the Neumann boundary condition (cf. \cite{Brandts2000}), we have
\begin{equation}\label{NeumDisHelm}
\bm{\xi}_{h}=\grad_{h}v_{h}\oplus\curlvect w_{h},
\end{equation}
where $(v_{h},w_{h})\in\mathcal{V}_{h}\times\mathcal{S}_{h},$ and $\grad_{h}v_{h}\in\mathcal{Q}_{0h}, w_{h}|_{\partial\Omega}=0$.
Then by following the proof of \cref{superclosep} and using \cref{lemmaNeumann} instead of \cref{lemmanolog}, we prove \cref{superclosepNeumann}.
\end{proof}

\section{Postprocessing operator and connection with CR nonconforming elements}\label{sec:post}
In the gradient recovery framework, once supercloseness
between the interpolant $\Pi_{h}\bm{p}$ and numerical solution $\bm{p}_{h}$ are available,
one can construct postprocessing operator $G_{h}$ to achieve superconvergence of $G_{h}\bm{p}_{h}$ to $\bm{p}$.
Of course the construction and analysis of $G_{h}$ is of independent interest (cf. \cite{ZZ1990,BaXu2003II,Xu2003,BaXuZheng2007}).
In this section, we first discuss a cheap recovery operator $G_{h}$ proposed in \cite{Brandts1994} and use it to achieve the superconvergence
\cref{superpost}. Then we prove the superconvergence estimate \cref{superpostCR} for CR nonconforming elements.

\subsection{Postprocessing operator}\label{subsec:Gh}
To define $G_{h}$, we introduce the nonconforming finite element space:
\begin{equation*}
\begin{aligned}
\mathcal{V}_{h}^{CR} := \{v:\ &v|_{\tau} \text{ is linear on } \tau\in\mathcal{T}_{h}, v
 \text{ is continuous}\\
&\text{at the midpoints of interior edges of } \mathcal{T}_{h}\}.
\end{aligned}
\end{equation*}

\begin{figure}[tbhp]
\centering
\includegraphics[width=10.0cm,height=5.0cm]{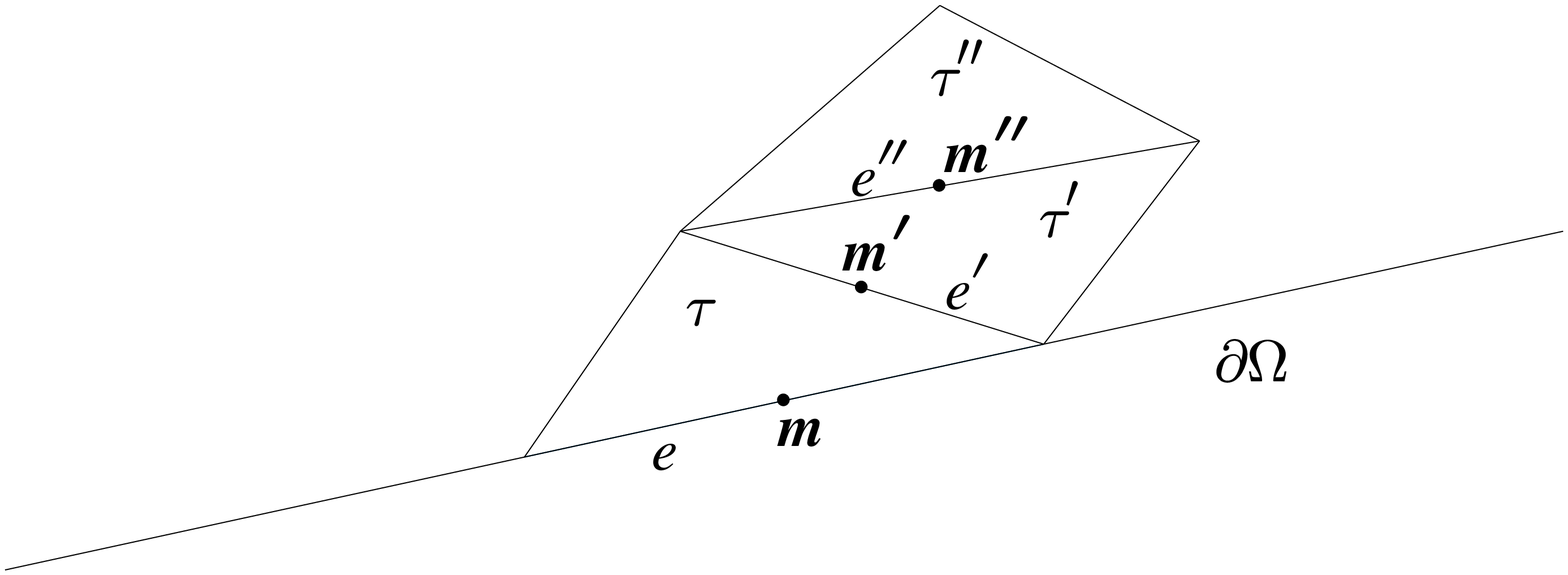}
\caption{a local patch $\omega$ near the boundary}
\label{bdygh}
\end{figure}

\begin{definition}\label{patch1}
The operator $G_{h}: \mathcal{Q}_{h}\rightarrow\mathcal{V}_{h}^{CR}\times\mathcal{V}_{h}^{CR}$ is defined as follows:
\begin{enumerate}
\item For each interior edge e, let $\tau$ and $\tau^{\prime}$ be the pair of elements sharing e. Then the value of
$G_{h}\bm{q}_{h}$ at the midpoint $\bm{m}$ of $e$ is
\begin{equation*}
G_{h}\bm{q}_{h}(\bm{m})=\frac{1}{2}(\bm{q}_{h}|_{\tau}(\bm{m})+\bm{q}_{h}|_{\tau^{\prime}}(\bm{m})).
\end{equation*}
\item For each boundary edge e$\subset\partial\Omega$, let $\tau$ be the element having e
as an edge. Let $\tau^{\prime}$ be an element sharing an edge $e^{\prime}$ with $\tau$.
Let $e^{\prime\prime}$ denote the edge of $\tau^{\prime}$ that does not meet with
$e$, $\tau^{\prime\prime}$ the element sharing $e^{\prime\prime}$ with $\tau^{\prime}$.
Then the value of $G_{h}\bm{q}_{h}$ at the midpoint $\bm{m}$ of $e$ is
\begin{equation*}
G_{h}\bm{q}_{h}(\bm{m})=2G_{h}\bm{q}_{h}(\bm{m}^{\prime})-G_{h}\bm{q}_{h}(\bm{m}^{\prime\prime}),
\end{equation*}
where $\bm{m}^{\prime}$ and $\bm{m}^{\prime\prime}$ is the midpoint of $e^{\prime}$ and $e^{\prime\prime}$, 
respectively, see \cref{bdygh}.
\end{enumerate}
\end{definition}

We have the following lemma.
\begin{lemma}\label{patchestimate}
Let $\omega$ be the patch $\tau\cup\tau^{\prime}$ associated with the interior edge or
$\tau\cup\tau^{\prime}\cup\tau^{\prime\prime}$ associated with the boundary edge having the midpoint $\bm{m}$ in \cref{patch1}.
Assume each pair of adjacent elements in $\omega$ forms an $\mathcal{O}(h^{1+\alpha})$ approximate parallelogram. Then we have
\begin{equation*}
|(\bm{q}_{L}-G_{h}\Pi_{h}\bm{q}_{L})(\bm{m})|\lesssim h^{\alpha}||\nabla\bm{q}_{L}||_{0,\omega},
\end{equation*}
where $\bm{q}_{L}\in\mathcal{P}_{1}(\omega)^{2}$.
\end{lemma}
\begin{proof}
First consider the case of interior edges. By the fact that $\Pi_{h}\bm{q}_{C}=\bm{q}_{C}$ for
$\bm{q}_{C}\in\mathcal{P}_{0}(\omega)$, we can assume $\bm{q}_{L}(\bm{m})=0$ without loss of generality.
Let $\bm{m}_{k}$ be the midpoint of $e_{k}$. Then by \cref{Cartesian},
\begin{equation}
\begin{aligned}
&(G_{h}\Pi_{h}\bm{q}_{L}-\bm{q}_{L})(\bm{m})\\
&=\frac{1}{2}\sum_{k=1}^{3}\left(
\frac{\bm{m}-\bm{a}_{k}}{2|\tau|}\int_{e_{k}}\bm{q}_{L}\cdot\bm{n}_{k}
+\frac{\bm{m}-\bm{a}_{k}^{\prime}}{2|\tau^{\prime}|}\int_{e_{k}^{\prime}}\bm{q}_{L}\cdot\bm{n}_{k}^{\prime}\right)\\
&=\frac{1}{2}\sum_{k=1}^{3}\left(\frac{\bm{m}-\bm{a}_{k}}{2|\tau|}\ell_{k}\bm{q}_{L}(\bm{m}_{k})\cdot\bm{n}_{k}
+\frac{\bm{m}-\bm{a}_{k}^{\prime}}{2|\tau^{\prime}|}\ell_{k}^{\prime}\bm{q}_{L}(\bm{m}_{k}^{\prime})\cdot\bm{n}_{k}^{\prime}\right),
\end{aligned}
\end{equation}
where $\bm{m}_{k}$ is the midpoint of $e_{k}$. From
\begin{equation*}
\bm{q}_{L}(\bm{m}_{k})=\nabla\bm{q}_{L}(\bm{m})(\bm{m}_{k}-\bm{m}),\quad
\bm{q}_{L}(\bm{m}_{k}^{\prime})=\nabla\bm{q}_{L}(\bm{m})(\bm{m}_{k}^{\prime}-\bm{m}),
\end{equation*}
and the $\mathcal{O}(h^{1+\alpha})$ approximate parallelogram condition, it follows that
\begin{equation}\label{interior}
|(G_{h}\Pi_{h}\bm{q}_{L}-\bm{q}_{L})(\bm{m})|\lesssim h^{\alpha}||\nabla\bm{q}_{L}||_{0,\omega}.
\end{equation}
As for the case of boundary edge, again assume $\bm{q}_{L}(\bm{m}^{\prime})=0$ without loss of generality.
Then by 
\begin{equation*}
\bm{q}_{L}(\bm{m})=\nabla\bm{q}_{L}(\bm{m}^{\prime})(\bm{m}-\bm{m}^{\prime}),\quad
\bm{q}_{L}(\bm{m}^{\prime\prime})=\nabla\bm{q}_{L}(\bm{m}^{\prime})(\bm{m}^{\prime\prime}-\bm{m}^{\prime}),
\end{equation*}
the $\mathcal{O}(h^{1+\alpha})$ approximate parallelogram condition and \cref{interior}, we have
\begin{equation*}
\begin{aligned}
\left|(\bm{q}_{L}-G_{h}\Pi_{h}\bm{q}_{L})(\bm{m})\right|&=|\bm{q}_{L}(\bm{m})+G_{h}\Pi_{h}\bm{q}_{L}(\bm{m}^{\prime\prime})
-2G_{h}\Pi_{h}\bm{q}_{L}(\bm{m}^{\prime})|\\
&\leq|\bm{q}_{L}(\bm{m})+\bm{q}_{L}(\bm{m}^{\prime\prime})|
+|G_{h}\Pi_{h}\bm{q}_{L}(\bm{m}^{\prime\prime})-\bm{q}_{L}(\bm{m}^{\prime\prime})|\\
&+2|G_{h}\Pi_{h}\bm{q}_{L}(\bm{m}^{\prime})-\mathbf{q}_{L}(\bm{m}^{\prime})|
\lesssim h^{\alpha}||\nabla\bm{q}_{L}||_{0,\omega}.
\end{aligned}
\end{equation*}
\end{proof}
\begin{theorem}\label{mainthminterp}
Assume the triangulation $\mathcal{T}_{h}$ satisfies the $(\alpha,\sigma)$-condition. Then
\begin{equation*}
||\bm{q}-G_{h}\Pi_{h}\bm{q}||_{0,\Omega}\lesssim h^{1+\rho}
(||\nabla\bm{q}||_{1,\Omega}+|\bm{q}|_{1,\infty,\Omega}).
\end{equation*}
\end{theorem}
\begin{proof}
Because $G_{h}$ is defined locally, we only need to estimate $\bm{q}-G_{h}\Pi_{h}\bm{q}$ element by element.
We partition the domain into three disjoint parts $\{\Omega_{i}\}_{i=1}^{3}$. $\Omega_{1}$ is covered by interior elements whose three edges belong to $\mathcal{E}_{1}$. $\Omega_{2}$ is covered by boundary elements $\tau$ that forms an approximate
parallelogram with one of its interior adjacent element $\tau^{\prime}$ and $\tau^{\prime}$ forms an approximate parallelogram with one of its interior adjacent elements $\tau^{\prime\prime}$, as in \cref{patch1}, see the pattern \cref{bdygh}. $\Omega_{3}$ is the complement of $\Omega_{1}\cup\Omega_{2}$. Then
\begin{equation}\label{split}
||\bm{q}-G_{h}\Pi_{h}\bm{q}||_{0,\Omega}^{2}
=\sum_{i=1}^{3}\sum_{\tau\subset\Omega_{i}}||\bm{q}-G_{h}\Pi_{h}\bm{q}||_{0,\tau}^{2}
=\sum_{i=1}^{3}\Rom{1}_{i}.
\end{equation}
For each element $\tau$, let $\tilde{\tau}$ denote the union of elements sharing a side with $\tau$.
For $\tau\subset\Omega_{1}$ or $\Omega_{2}$, $||\bm{q}-G_{h}\Pi_{h}\bm{q}||_{0,\tau}$ is estimated by passing through a linear polynomial
$\bm{q}_{L}\in P_{1}({\tilde{\tau}})^{2}$:
\begin{equation}\label{splitinterp}
\begin{aligned}
||\bm{q}-G_{h}\Pi_{h}\bm{q}||_{0,\tau}&\lesssim||\bm{q}-\bm{q}_{L}||_{0,\tau}\\
&+||G_{h}\Pi_{h}(\bm{q}-\bm{q}_{L})||_{0,\tau}+||\bm{q}_{L}-G_{h}\Pi_{h}\bm{q}_{L}||_{0,\tau}.
\end{aligned}
\end{equation}
By the Bramble--Hilbert lemma and scaling argument, there exists $\bm{q}_{L}\in\mathcal{P}_{1}(\tilde{\tau})^{2}$ such that
\begin{equation}\label{1stinterp}
||\bm{q}-\bm{q}_{L}||_{s,\tilde{\tau}}\lesssim h^{2-s}|\bm{q}|_{2,\tilde{\tau}},\quad s=0, 1,
\end{equation}
and
\begin{equation}\label{3rdinterp}
\begin{aligned}
||G_{h}\Pi_{h}(\bm{q}-\bm{q}_{L})||_{0,\tau}&\lesssim h||G_{h}\Pi_{h}(\bm{q}-\bm{q}_{L})||_{0,\infty,\tau}\\
&\lesssim h||\bm{q}-\bm{q}_{L}||_{0,\infty,\tilde{\tau}}\lesssim h^{2}|\bm{q}|_{2,\tilde{\tau}}.
\end{aligned}
\end{equation}
Then by \cref{patchestimate,splitinterp,1stinterp,3rdinterp}, we have
\begin{equation}\label{omega12}
\begin{aligned}
\Rom{1}_{1}+\Rom{1}_{2}&\lesssim\sum_{i=1}^{2}\sum_{\tau\subset\Omega_{i}}\left\{h^{4}|\bm{q}|_{2,\tilde{\tau}}^{2}
+h^{2}||\bm{q}_{L}-G_{h}\Pi_{h}\bm{q}_{L}||_{0,\infty,\tau}^{2}\right\}\\
&\lesssim\sum_{i=1}^{2}\sum_{\tau\subset\Omega_{i}}\left\{h^{4}|\bm{q}|_{2,\tilde{\tau}}^{2}
+h^{2}\max_{1\leq k\leq3}(\bm{q}_{L}(\bm{m}_{k})-G_{h}\Pi_{h}\bm{q}_{L}(\bm{m}_{k}))^{2}\right\}\\
&\lesssim h^{2+2\min(1,\alpha)}\sum_{i=1}^{2}\sum_{\tau\subset\Omega_{i}}(|\bm{q}|_{2,\tilde{\tau}}^{2}
+||\nabla\bm{q}_{L}||_{\tilde{\tau}}^{2})\\
&\lesssim h^{2+2\min(1,\alpha)}||\nabla\bm{q}||_{1,\Omega}^{2}.
\end{aligned}
\end{equation}
By the $(\alpha,\sigma)$-condition and local quasi-uniformity of $\mathcal{T}_{h}$,
$|\Omega_{3}|$ is forced to be of the size $\mathcal{O}(h^{\sigma})$. Since $G_{h}\Pi_{h}\bm{q}_{C}=\bm{q}_{C}$ for 
$\bm{q}_{C}\in\mathcal{P}_{0}(\tilde{\tau})^{2}$, $\Rom{1}_{3}$ can be estimated by the Bramble--Hilbert lemma with a scaling argument
and the small measure of $\Omega_{3}$:
\begin{equation}\label{omega3}
\begin{aligned}
\Rom{1}_{3}&=\sum_{\tau\subset\Omega_{3}}||\bm{q}-G_{h}\Pi_{h}\bm{q}||_{0,\tau}^{2}\\
&\lesssim\sum_{\tau\subset\Omega_{3}}h^{2}|\bm{q}|_{1,\tilde{\tau}}^{2}\lesssim h^{2}\int_{\Omega_{3}}|\nabla\bm{q}|^{2}
\lesssim h^{2+\sigma}|\bm{q}|_{1,\infty,\Omega}^{2}.\\
\end{aligned}
\end{equation}
By \cref{split,omega12,omega3}, we obtain \cref{mainthminterp}.
\end{proof}

The supercovnergence of $||\bm{p}-G_{h}\bm{p}_{h}||_{0,\Omega}$ is a direct result from \cref{superclosep,mainthminterp}.
\begin{theorem}\label{finaltheorem}
Let $\mathcal{T}_{h}$ be quasi-uniform and satisfy the $(\alpha,\sigma)$-condition. Then
\begin{equation*}
||\bm{p}-G_{h}\bm{p}_{h}||_{0,\Omega}\lesssim h^{1+\rho}||u||_{4+\varepsilon,\Omega}.
\end{equation*}
\end{theorem}
\begin{proof}
For $\bm{q}_{h}\in\mathcal{Q}_{h}$ and $\tau\in\mathcal{T}_{h}$,
\begin{equation*}
\begin{aligned}
||G_{h}\bm{q}_{h}||_{0,\tau}&\lesssim h||G_{h}\bm{q}_{h}||_{0,\infty,\tau}\\
&\lesssim h\max_{1\leq k\leq3}|G_{h}\bm{q}_{h}(\bm{m}_{k})|\\
&\lesssim h||\bm{q}_{h}||_{0,\infty,\tilde{\tau}}\lesssim||\bm{q}_{h}||_{0,\tilde{\tau}},
\end{aligned}
\end{equation*}
and then
\begin{equation}\label{bdgh}
||G_{h}\bm{q}_{h}||_{0,\Omega}\lesssim||\bm{q}_{h}||_{0,\Omega},
\end{equation}
that is, $G_{h}$ is bounded in $L^{2}$ norm.
Combining \cref{superclosep,mainthminterp,bdgh}, we have
\begin{equation*}
\begin{aligned}
||\bm{p}-G_{h}\bm{p}_{h}||_{0,\Omega}&\lesssim||\bm{p}-G_{h}\Pi_{h}\bm{p}_{h}||_{0,\Omega}
+||G_{h}(\Pi_{h}\bm{p}-\bm{p}_{h})||_{0,\Omega}\\
&\lesssim h^{1+\rho}||u||_{4+\varepsilon,\Omega}.
\end{aligned}
\end{equation*}
\end{proof}

\subsection{Superconvergence for CR nonconforming elements}\label{subsec:CR}
As one can see in \cref{sec:varerr,subsec:Gh}, the CR nonconforming method is
closely related to the RT mixed method. In this subsection, we prove the superconvergence
estimate for CR nonconforming methods on $(\alpha,\sigma)$-grids by \cref{finaltheorem}. 
For simplicity, we only consider Poisson's equation with the homogeneous
Dirichlet boundary condition, i.e. $\bm{A}=\bm{I}_{2\times2}, \bm{b}=\bm{0}, c=0, g=0$ in \cref{mix:c1}.
In this case, $\bm{p}=\nabla u$ in \cref{mix:c}. 
The corresponding CR nonconforming method is to find $u_{h}^{CR}\in\mathcal{V}_{0h}^{CR}$, such that
\begin{equation}\label{CR}
(\nabla_{h}{u}_{h}^{CR},\nabla_{h}v_{h})=(f,v_{h}),\quad v_{h}\in\mathcal{V}_{0h}^{CR},
\end{equation}
where
\begin{equation*}
\mathcal{V}^{CR}_{0h} := \{v\in\mathcal{V}^{CR}_{h} : v=0 \text{ at the midpoints of boundary edges of } \mathcal{T}_{h}.\}.
\end{equation*}

Marini \cite{Marini1985} proved the following theorem.
\begin{theorem}\label{Mariniequiv}
Let $\bar{u}_{h}^{CR}\in\mathcal{V}_{0h}^{CR}$ solve
\begin{equation}\label{CRbar}
(\nabla_{h}\bar{u}_{h}^{CR},\nabla_{h}v_{h})=(P_{h}f,v_{h}),\quad v_{h}
\in\mathcal{V}_{0h}^{CR}.
\end{equation}
Let $\{\bar{\bm{p}}_{h},\bar{u}_{h}\}\in\mathcal{Q}_{h}\times\mathcal{V}_{h}$
solve the following mixed problem:
\begin{subequations}
\begin{align}
&(\bar{\bm{p}}_{h},\bm{q}_{h})+(\divg\bm{q}_{h},\bar{u}_{h})=0,\quad\bm{q}_{h}\in\mathcal{Q}_{h},\nonumber\\
&(\divg\bar{\bm{p}}_{h},v_{h})=-(f,v_{h}),\quad v_{h}\in\mathcal{V}_{h}.\nonumber
\end{align}
\end{subequations}
Then
\begin{equation*}
\bar{\bm{p}}_{h}(\bm{x})=\nabla\bar{u}_{h}^{CR}-\frac{P_{h}f|_{\tau}}{2}(\bm{x}-\bm{x}_{\tau}),\quad\bm{x}\in\tau,
\end{equation*}
where $\bm{x}_{\tau}$ is the barycenter of $\tau\in\mathcal{T}_{h}$.
\end{theorem}

By \cref{Mariniequiv,resultB}, Hu and Ma \cite{HuMa2016} proved the following estimate for Poisson's equation on uniform grids:
\begin{equation}\label{resultHM}
||\nabla u-G_{h}\nabla_{h}u_{h}^{CR}||_{0,\Omega}\lesssim h^{\frac{3}{2}}(||u||_{\frac{5}{2},\Omega}
+h^{\frac{1}{2}}|u|_{3,\Omega}+h^{\frac{1}{2}}|f|_{1,\infty,\Omega}).
\end{equation}
Based on the idea of \cite{HuMa2016}, we can prove the following superconvergence estimate for CR elements by \cref{finaltheorem,Mariniequiv}.
\begin{theorem}\label{superCR}
Let $\mathcal{T}_{h}$ be a quasi-uniform $(\alpha,\sigma)$-grid. Let $u_{h}^{CR}$ solve \cref{CR}. Then
\begin{equation*}
||\nabla u-G_{h}\nabla_{h}u_{h}^{CR}||_{0,\Omega}\lesssim h^{1+\rho}||u||_{4+\varepsilon,\Omega}.
\end{equation*}
\end{theorem}
\begin{proof}
Split $||\nabla u-G_{h}\nabla_{h}u_{h}^{CR}||_{0,\Omega}$ as
\begin{equation}\label{splitCR}
\begin{aligned}
&||\nabla u-G_{h}\nabla_{h}u_{h}^{CR}||_{0,\Omega}\\
\lesssim&||\nabla u-G_{h}\bar{\bm{p}}_{h}||_{0,\Omega}
+||G_{h}(\bar{\bm{p}}_{h}-\nabla_{h}\bar{u}_{h}^{CR})||_{0,\Omega}\\
+&||G_{h}(\nabla_{h}\bar{u}_{h}^{CR}-\nabla_{h}u_{h}^{CR})||_{0,\Omega}=\Rom{1}+\Rom{2}+\Rom{3}.
\end{aligned}
\end{equation}
$\Rom{1}$ can be estimated by \cref{finaltheorem}:
\begin{equation}\label{bdICR}
\Rom{1}\lesssim h^{1+\rho}||u||_{4+\varepsilon,\Omega}.
\end{equation}
For the second term, first consider the patch $\omega=\tau\cup\tau^{\prime}$ associated with an interior edge
$e$ having midpoint $\bm{m}$ in \cref{patch1}. By \cref{Mariniequiv}, we have
\begin{equation*}
G_{h}(\bar{\bm{p}}_{h}-\nabla_{h}\bar{u}_{h}^{CR})(\bm{m})=-\frac{1}{4}\left((\bm{m}-\bm{x}_{\tau})P_{h}f|_{\tau}
+(\bm{m}-\bm{x}_{\tau^{\prime}})P_{h}f|_{\tau^{\prime}}\right)
\end{equation*}
If $\tau$ and $\tau^{\prime}$ form an $\mathcal{O}(h^{1+\alpha})$ approximate parallelogram, then
\begin{equation}\label{interiormid}
|G_{h}(\bar{\bm{p}}_{h}-\nabla_{h}\bar{u}_{h}^{CR})(\bm{m})|\lesssim h^{1+\alpha}||f||_{1,\infty,\Omega}.
\end{equation}
Then consider the patch $\omega=\tau\cup\tau^{\prime}\cup\tau^{\prime\prime}$ associated with a boundary edge
$e$ having midpoint $\bm{m}$ in \cref{patch1}. \cref{interiormid} implies that
\begin{equation}\label{bdymid}
\begin{aligned}
|G_{h}(\bar{\bm{p}}_{h}-\nabla_{h}\bar{u}_{h}^{CR})(\bm{m})|
&\lesssim2|G_{h}(\bar{\bm{p}}_{h}-\nabla_{h}\bar{u}_{h}^{CR})(\bm{m}^{\prime})|\\
&+|G_{h}(\bar{\bm{p}}_{h}-\nabla_{h}\bar{u}_{h}^{CR})(\bm{m}^{\prime\prime})|
\lesssim h^{1+\alpha}||f||_{1,\infty,\Omega},
\end{aligned}
\end{equation}
provided $\tau,\tau^{\prime}$ and $\tau^{\prime},\tau^{\prime\prime}$ form $\mathcal{O}(h^{1+\alpha})$ approximate parallelograms.
Now we partition $\Omega$ into $\cup_{i=1}^{3}\Omega_{i}$ as in the proof of \cref{mainthminterp}.
By following the proof of \cref{mainthminterp} and using \cref{interiormid,bdymid}, $\Rom{2}$ can be estimated by
\begin{equation}\label{bdIICR}
\Rom{2}\lesssim h^{1+\rho}||f||_{1,\infty,\Omega}.
\end{equation}
As for $\Rom{3}$, it follows from \cref{CR,CRbar} that for $v_{h}\in\mathcal{V}_{0h}^{CR}$,
\begin{equation}\label{varCRCRbar}
\begin{aligned}
(\nabla_{h}\bar{u}_{h}^{CR}-\nabla_{h}u_{h}^{CR},\nabla_{h}v_{h})&=(f-P_{h}f,v_{h})\\
&=(f-P_{h}f,v_{h}-P_{h}v_{h})\\
&\lesssim h^{2}|f|_{1,\Omega}||\nabla_{h}v_{h}||_{0,\Omega},
\end{aligned}
\end{equation}
By setting $v_{h}=\bar{u}_{h}^{CR}-u_{h}^{CR}$ in \cref{varCRCRbar} and using boundedness 
of $G_{h}$ in $L^{2}$ norm, we have 
\begin{equation}\label{bdIIICR}
\Rom{3}\lesssim h^{2}|f|_{1,\Omega}.
\end{equation}
Then \cref{superCR} results from combining \cref{splitCR,bdICR,bdIICR,bdIIICR}.
\end{proof}

In the case of uniform grids $(\alpha=\sigma=\infty)$, \cref{superCR} implies that
\begin{equation*}
||\nabla u-G_{h}\nabla_{h}u_{h}^{CR}||_{0,\Omega}\lesssim h^{2}||u||_{4+\varepsilon,\Omega}.
\end{equation*}
which shows that is \cref{resultHM} suboptimal.
However, \cref{Mariniequiv} cannot be applied to \cref{mix:c1} with nonvanishing $\bm{b}$ and $c$.
It would be interesting to develop a formula similar  to \cite{Marini1985}  in a more general setting.

\section{Numerical experiments}\label{sec:numerexp}
\begin{figure}[tbhp]
\centering
\begin{tabular}[c]{ccccc}%
  \subfigure[]{\label{grida}\includegraphics[width=4.0cm,height=4.0cm]{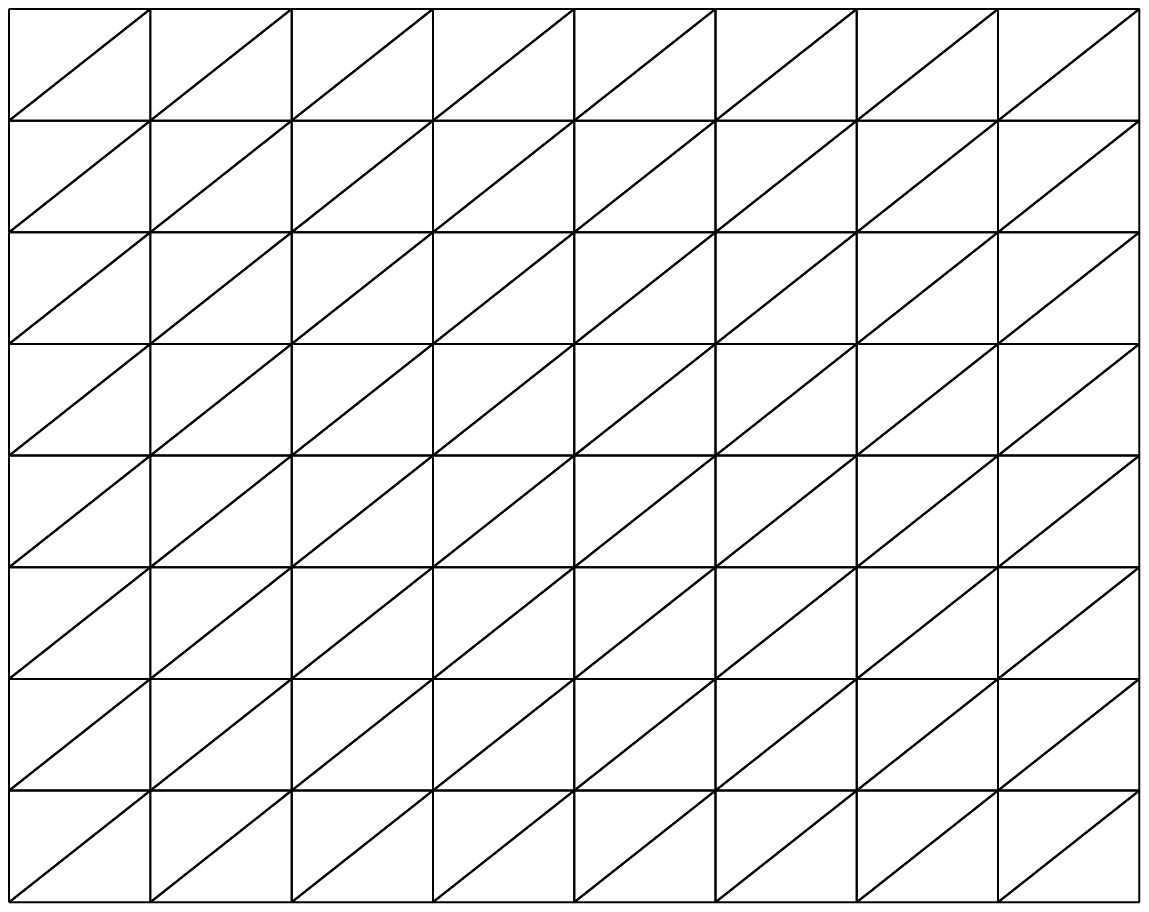}}
  \subfigure[]{\label{gridb}\includegraphics[width=4.0cm,height=4.0cm]{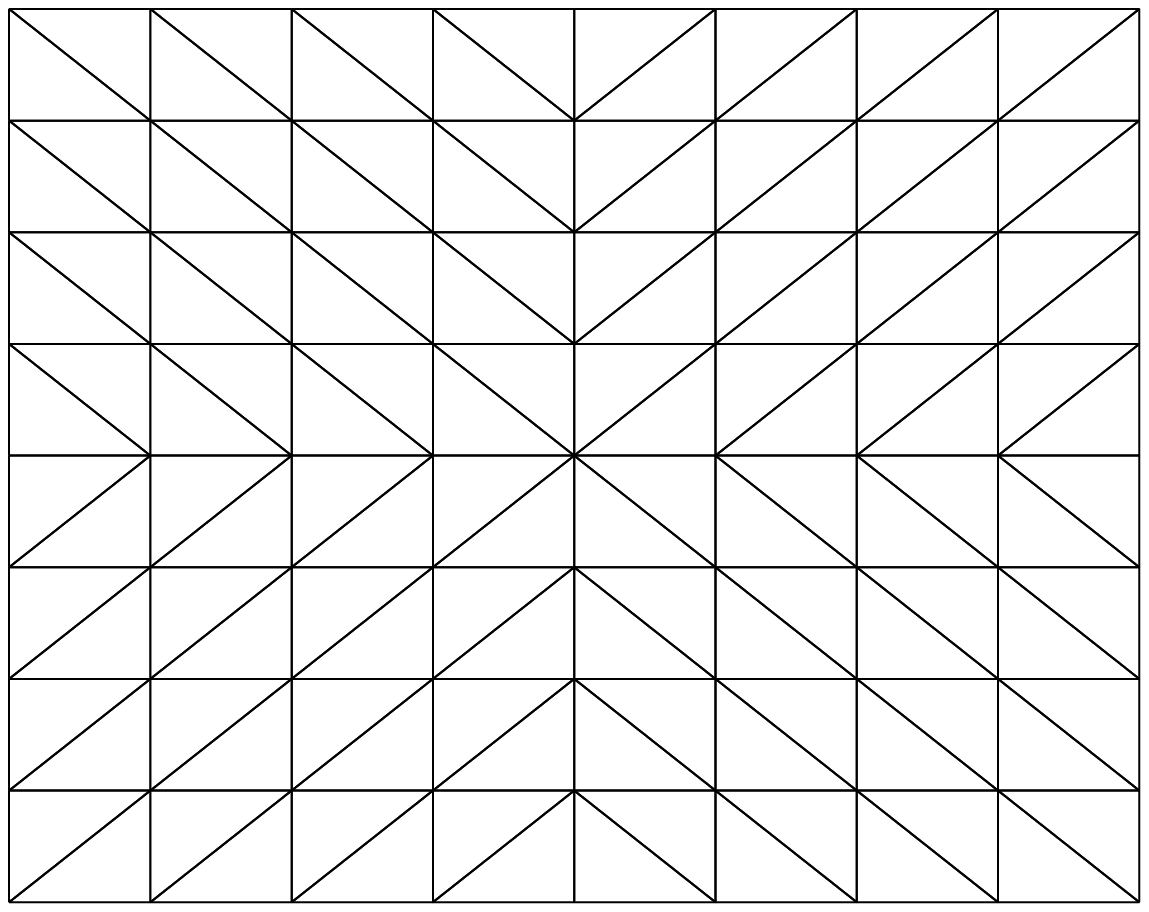}}
  \subfigure[]{\label{gridc}\includegraphics[width=4.0cm,height=4.0cm]{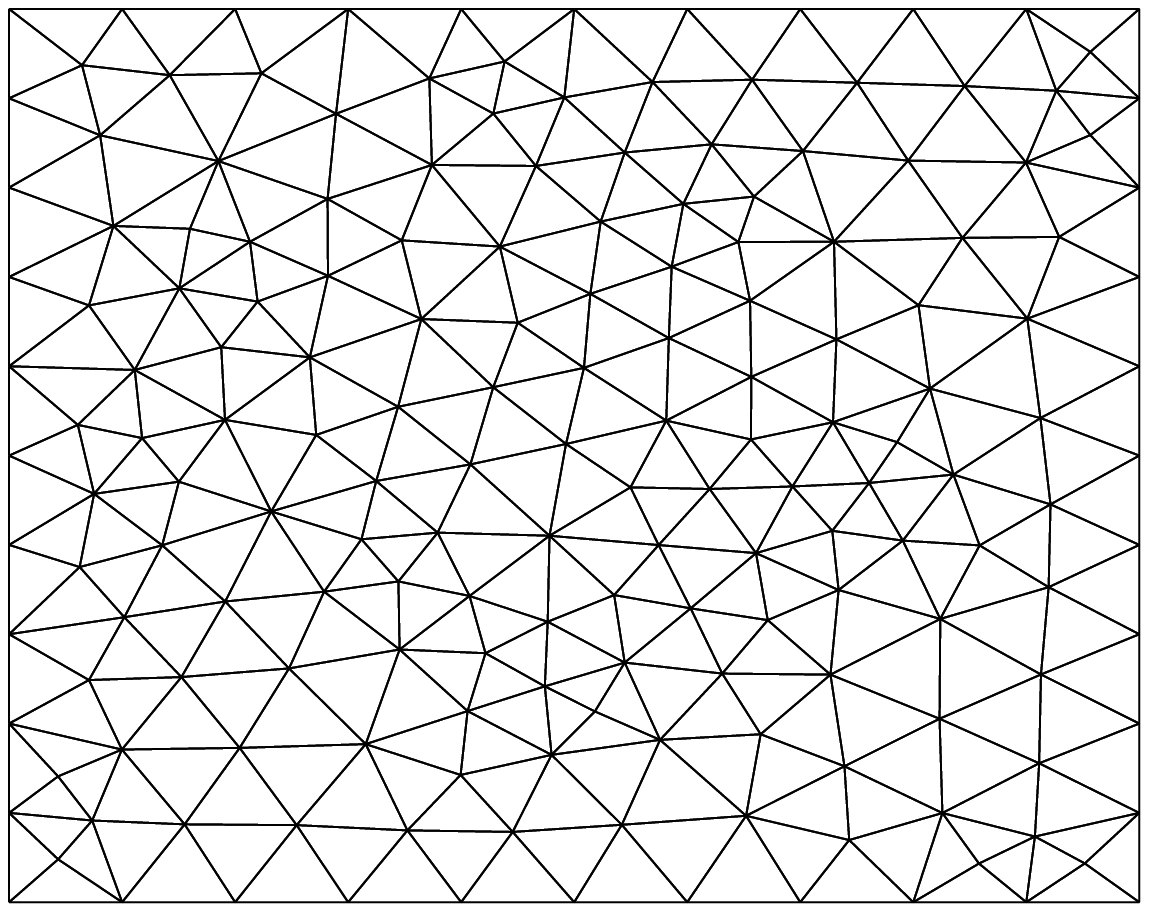}}
\end{tabular}
\caption{Three different grids}
\label{grid}
\end{figure}

\begin{table}[tbhp]
\caption{Uniform grids}
\label{tableu}
\centering
\begin{tabular}{|c|c|c|c|c|c|c|}
\hline
\multirow{2}*{nu} &  \multicolumn{2}{c|}{$||\bm{p}-\bm{p}_{h}||_{0,\Omega}$}
 & \multicolumn{2}{c|}{$||\Pi_{h}\bm{p}-\bm{p}_{h}||_{0,\Omega}$}
&  \multicolumn{2}{c|}{$||\bm{p}-G_{h}\bm{p}_{h}||_{0,\Omega}$} \\
\cline{2-7}
               &   Error &   Order  & Error & Order & Error & Order \\
\hline
               336   & 7.281e-1 & 0.9911 & 1.033e-1 & 1.979 & 2.629e-1 & 2.094 \\
             1312       & 3.663e-1 & 0.9972 & 2.620e-2 &  1.995 & 6.157e-2 & 2.062  \\
             5184         & 1.835e-1 & 0.9998 & 6.574e-3 & 1.999 & 1.475e-2 & 2.035 \\
               20608       & 9.176e-2 & 0.9997 & 1.645e-3 & 1.999 & 3.598e-3 & 2.003 \\
               82176       & 4.589e-2 &            &  4.114e-4  &       & 8.976e-4 &   \\
\hline
\end{tabular}
\end{table}

\begin{table}[tbhp]
\caption{Piecewise uniform grids}
\label{tablecc}
\centering
\begin{tabular}{|c|c|c|c|c|c|c|}
\hline
\multirow{2}*{nu} &  \multicolumn{2}{c|}{$||\bm{p}-\bm{p}_{h}||_{0,\Omega}$}
 & \multicolumn{2}{c|}{$||\Pi_{h}\bm{p}-\bm{p}_{h}||_{0,\Omega}$}
&  \multicolumn{2}{c|}{$||\bm{p}-G_{h}\bm{p}_{h}||_{0,\Omega}$} \\
\cline{2-7}
               &   Error &   Order  & Error & Order & Error & Order \\
\hline
               336   & 7.287e-1 & 0.9919 & 9.356e-2 & 1.937 & 2.898e-1 & 1.963 \\
             1312       & 3.664e-1 & 0.9976 & 2.449e-2 &  1.978 & 7.668e-2 & 1.790  \\
             5184         & 1.835e-1 & 0.9998 & 6.215e-3 & 1.993 & 2.217e-2 & 1.683 \\
               20608       & 9.176e-2 & 0.9997 & 1.561e-3 & 1.998 & 6.904e-3 & 1.607 \\
               82176       & 4.589e-2 &            &  3.907e-4  &          & 2.267e-3 &   \\
\hline
\end{tabular}
\end{table}

\begin{table}[tbhp]
\caption{Unstructured grids}
\label{tablems}
\centering
\begin{tabular}{|c|c|c|c|c|c|c|}
\hline
\multirow{2}*{nu} &  \multicolumn{2}{c|}{$||\bm{p}-\bm{p}_{h}||_{0,\Omega}$}
 & \multicolumn{2}{c|}{$||\Pi_{h}\bm{p}-\bm{p}_{h}||_{0,\Omega}$}
&  \multicolumn{2}{c|}{$||\bm{p}-G_{h}\bm{p}_{h}||_{0,\Omega}$} \\
\cline{2-7}
               &   Error &   Order  & Error & Order & Error & Order \\
\hline
               800   & 4.585e-1 & 0.9934 & 5.152e-2 & 1.749 & 1.911e-1 & 1.724 \\
             3160       & 2.303e-1 & 0.9981 & 1.533e-2 &  1.823 & 5.783e-2 & 1.551  \\
             12560         & 1.153e-1 & 0.9990 & 4.334e-3 & 1.862 & 1.973e-2 & 1.532 \\
               50080       & 5.769e-2 & 0.9997 & 1.192e-3 & 1.885 & 6.824e-3 & 1.518 \\
              200000       & 2.885e-2 &             & 3.227e-4 &            & 2.383e-3  &          \\
\hline
\end{tabular}
\end{table}

In this section, we test our superconvergence results for $||\Pi_{h}\bm{p}-\bm{p}_{h}||_{0,\Omega}$ and
$||\bm{p}-G_{h}\bm{p}_{h}||_{0,\Omega}$ by the following equation:
\begin{equation*}
\begin{aligned}
&-\Delta u+u=f,\quad\bm{x}\in\Omega,\\
&u=0,\quad\bm{x}\in\partial\Omega,
\end{aligned}
\end{equation*}
where $\Omega=(0,1)\times(0,1)$ is the unit square. Let $u=\sin(2\pi x_{1})\sin(\pi x_{2})$
and $f$ be the corresponding source term. The numerical experiments were performed using MATLAB, R2016.
The linear system resulting from the mixed method \cref{mix:dv} was solved by the operation $\backslash$.
The `nu' in \cref{tableu,tablecc,tablems} stands for the number of unknowns.

We began with the $8\times8$ uniform grid in \cref{grida}, and
computed a sequence of meshes by regular refinement, i.e. partitioning an element into four similar subelements by connecting
the midpoints of each edge. In this case, $\alpha=\sigma=\infty$, $\rho=1$. As shown in \cref{tableu}, the observed orders of convergence
coincide with \cref{superclosep} and \cref{finaltheorem}.

Then we considered the $(\alpha,\sigma)$-grid with $(\alpha,\sigma)=(\infty,1)$ in \cref{gridb}. The mesh was refined regularly.
Although it is not globally uniform, it can be decomposed into four uniform subgrids. Hence the mesh is a piecewise $(\alpha,\sigma)$-grid with $(\alpha,\sigma)=(\infty,\infty)$. By \cref{superclosep},
we still obtain 2nd order of convergence for $||\Pi_{h}\bm{p}-\bm{p}_{h}||_{0,\Omega}$, which was confirmed
by \cref{tablecc}. However, \cref{finaltheorem} cannot be applied to piecewise $(\alpha,\sigma)$-grid. Thus the order 
of convergence for $||\bm{p}-G_{h}\bm{p}_{h}||_{0,\Omega}$ approaches $3/2$ in \cref{tablecc}.

In the last experiment, we generated the initial mesh in \cref{gridc} by `pdetool' in MATLAB and then refined it regularly.
At first glance, it should be an unstructured grid or a mildly structured grid with unknown $\alpha$ and $\sigma$. Surprisingly,
it is indeed a piecewise uniform mesh, since each element in the initial mesh was refined uniformly. On the other hand,
it's not hard to see that the sequence of grids are globally $(\alpha,\sigma)$-meshes with $\alpha=\infty$ and $\sigma=1$ 
asymptotically. As predicted by \cref{superclosep,finaltheorem}, the order of convergence for
$||\Pi_{h}\bm{p}-\bm{p}_{h}||_{0,\Omega}$ approaches 2 while the order of convergence for $||\bm{p}-G_{h}\bm{p}_{h}||_{0,\Omega}$ approaches 3/2. We didn't obtain exact 2nd order convergence for  $||\Pi_{h}\bm{p}-\bm{p}_{h}||_{0,\Omega}$, since the linear system eventually became extremely large. In this case, both time cost and the condition number of the coefficient matrix became unacceptable.

\section{Concluding remarks}
In this paper, we proved optimal order global superconvergence for the lowest order RT element on mildly
structured meshes for general second order elliptic equations. As a byproduct, we also proved superconvergence for the CR nonconforming
method for Poisson's equations. The results in \cref{sec:preli,sec:varerr} are
PDE-independent and applicable to other numerical PDEs using lowest order RT elements. The proof of \cref{superclosedivp} and most steps of the proof of \cref{superclosep} work for higher order mixed finite elements.

In practice, the solution $u$ does not necessarily belong to $W^{3}_{\infty}(\Omega)$ in \cref{mainlemma} or $H^{4+\varepsilon}(\Omega)$ in \cref{superclosep} if $\partial\Omega$ is not smooth enough. Hence our superconvergence estimates become questionable in this case. In fact, the high regularity requirement is a common issue shared by most superconvergence results 
(cf. \cite{BaXu2003,Brandts1994,Ewing1991,HuangXu2008,Xu2003}). There are several possible ways to fix it. First, by modifying the proof of \cref{mainlemma}, one can obtain smaller rate of superconvergence under weaker regularity assumptions. For example, one can easily show $||\bm{p}-G_{h}\bm{p}_{h}||_{0,\Omega}=\mathcal{O}(h^{1+\min(1/2,\alpha,\sigma/2)})$ on $(\alpha,\sigma)$-grids for $u\in H^{3}(\Omega)\cap W_{\infty}^{2}(\Omega)$ if the error occuring on boundary triangles is not canceled in the proof of \cref{mainlemma}. Second, $u$ is smooth on any compact subdomain in $\Omega$. Hence it's meaningful to look for interior estimates (cf. \cite{WahlbinSchatz1995,Wahlbin1995}). 
As far as we know, $u$ should be at least in 
$W_{\infty}^{2}(\Omega)$ to prove interior superconvergence for linear Lagrange elements (cf. \cite{Xu2003}). Of course, the assumption $u\in W_{\infty}^{2}(\Omega)$ may not hold on domains with corners. Third, for $u\in H^{1+\delta}$ with $\delta>0$, it's possible to prove superconvergence recovery for RT elements under adaptive meshes by following the framework of this paper and assuming certain mesh density function which is enough to resolve the singularity, see \cite{WuZhang2007} for the case of Lagrange elements. However, it's difficult to prove that the adaptively refined sequence of meshes actually satisfies the mesh density pattern.

We also point out that $||\Pi_{h}\bm{p}-\bm{p}_{h}||_{0,\Omega}$ might not superconverge in the case of general mixed elements on triangular meshes, since part of finite element basis functions become more localized in higher order methods (cf. \cite{Brandts2000,Li2004}). We will present superconvergence results for higher order mixed elements on mildly structured meshes in another paper.

\section*{Acknowledgments}
The author would like to thank Professor Randolph E. Bank, for his guidance and helpful suggestions pertaining to this work.

\bibliographystyle{siamplain}

\end{document}